\documentclass[11pt]{article}
\usepackage[utf8]{inputenc}
\usepackage[a4paper,left = 1.25in,right = 1.25in, top = 1.25in, bottom = 1.25in]{geometry}
\usepackage{amsmath}
\usepackage{amsthm}
\usepackage{bm}
\usepackage{bbm}
\usepackage{cancel}
\usepackage{amssymb}
\usepackage{dsfont}
\usepackage{commath}
\usepackage{mathtools}
\usepackage{float}
\usepackage{cases}
\usepackage{url}

\usepackage{enumitem,linegoal}

\usepackage[numbers]{natbib}

\def\be{\begin{equation}}
\def\ee{\end{equation}}

\def\bea{\begin{align}}
\def\eea{\end{align}}

\def\bea*{\begin{align*}}
\def\eea*{\end{align*}}
\numberwithin{equation}{section}

\theoremstyle{plain}
\newtheorem{theorem}{Theorem}[section]

\newtheorem{lemma}{Lemma}[section]
\newtheorem{corollary}{Corollary}[section]
\newtheorem{proposition}{Proposition}[section]

\theoremstyle{definition}
\newtheorem{definition}{Definition}[section]
\newtheorem{example}{Example}[section]
\newtheorem{remark}{Remark}[section]

\DeclareMathOperator{\C}{\mathbb{C}}

\DeclareMathOperator{\E}{\mathbb{E}}
\DeclareMathOperator{\F}{\mathbb{F}}

\DeclareMathOperator{\N}{\mathbb{N}}

\DeclareMathOperator{\Q}{\mathbb{Q}}
\DeclareMathOperator{\R}{\mathbb{R}}

\DeclareMathOperator{\calB}{\mathcal{B}}
\DeclareMathOperator{\calC}{\mathcal{C}}

\DeclareMathOperator{\calF}{\mathcal{F}}

\DeclareMathOperator{\calI}{\mathcal{I}}

\DeclareMathOperator{\calM}{\mathcal{M}}

\DeclareMathOperator{\calO}{\mathcal{O}}

\def\mo{\mathfrak{o}}

\usepackage{subcaption}
\usepackage{graphicx}
\usepackage{color}
\usepackage{algorithm}

\usepackage[usenames,dvipsnames]{xcolor}
\usepackage{hyperref}
\hypersetup{colorlinks,linkcolor=BlueViolet,citecolor=BlueViolet,urlcolor = magenta}
\usepackage[capitalise]{cleveref}

\usepackage{authblk}

\newtheorem{assumption}{Assumption}[section]

\title{\vspace{-10mm}
Cylindrical Projections of Occupied Diffusions 
\vspace{-3mm}}

\date{\today}
\author{Valentin Tissot-Daguette\thanks{Quantitative Research, Office of the  CTO, Bloomberg. Email: \texttt{vtissotdague@bloomberg.net}} \,  and Xin Zhang\thanks{Department of Finance and Risk Engineering, New York University.  Email: {\tt xz1662@nyu.edu}. X. Zhang is partially supported by the NSF Grant DMS-2508556.}}

\begin{document}
\maketitle

\begin{abstract}
Occupied diffusions offer a Markovian framework for path-dependent dynamics by lifting  the state space with a flow of occupation measures. Because this additional feature is  infinite-dimensional, the simulation of these processes  remains computationally intractable. We address this by introducing \emph{cylindrical projections}, which approximate the  occupation flow via a  finite-dimensional system. We establish the strong convergence of this approximation to the initial process and derive corresponding convergence rates. The method is validated through Euler--Maruyama simulations of self-interacting  diffusions and an application to the Local Occupied Volatility (LOV) model in finance. 
Finally, we  provide a weak error analysis and explore its consequences for Monte Carlo methods and derivatives pricing.
\end{abstract}

{\small
\noindent\textbf{Keywords.} Occupation measure, cylindrical projections, error analysis, stochastic simulation, self-interacting diffusions, path-dependent volatility}

\vspace{3mm}

\noindent\textbf{AMS 2020 Subject Classification.} 
65C30, 
60J25,  
62C05, 
65Z05 

\section{Introduction}

Many stochastic dynamics arising in finance, engineering, and biology  are
\emph{path-dependent}: the present evolution depends not only on the current state, but also on the past trajectory. They appear naturally across applications: in finance for volatility models with history-dependent dynamics \cite{GrasselliPages,Guyon2014,GuyonSlides,GuyonLekeufack}; in engineering and biology for delayed dynamics and systems \cite{GSBOM16,WWQ16}. To capture  path-dependent phenomena, a natural approach  is to work on the path space directly  \cite{Dupire19}, enabling the analysis of such systems via path-dependent partial differential equations (PPDEs)   \cite{EKTZ14,ETZ16}. While this framework offers remarkable generality, it entails considerable computational complexity and analytical challenges around the regularity of path functionals, such as solutions of PPDEs.

An alternative, proposed in \cite{TissotOP,TissotThesis} (see also \cite{Bethencourt,SonerTissotZhang}), consists of enlarging the state  space by a 
measure-valued variable, the \textit{occupation flow},  which records the time spent by the process
in arbitrary regions. One then recovers a Markovian framework that captures a specific form of path dependence arising from  occupation times. 
 In this work, we shall focus on   
\emph{occupied diffusions}, solving stochastic differential equations (SDEs) of the form 
\begin{equation}\label{eq:introOSDE}
\begin{cases}
d\calO_t = \delta_{X_t}\,\lambda(\calO_t,X_t)\,dt,\\[0.3em]
dX_t = b(\calO_t,X_t)\,dt + \sigma(\calO_t,X_t)\,dW_t,
\end{cases}
\end{equation}
where $X_t\in \R^d$ and $\calO_t$ is a  finite measure on $\R^d$
that  encodes the spatial distribution of the process's history. With  $\calO_0=0$, $\lambda \equiv 1$, and $A \subseteq \mathbb R^d$, $\calO_t(A)=| \{s \in [0,t]: \, X_s \in A \}|$ is the amount of time $X$ has spent in $A$ up to time $t$. The well-posedness of $\eqref{eq:introOSDE}$ was established in \cite{TissotOP}, alongside several financial  applications. 
As argued in  \cite{TissotLOV}, the path dependence of the diffusion coefficient $\sigma$ (through $\calO$) gives rise to a  subclass of path-dependent volatility models \cite{GrasselliPages,Guyon2014,GuyonSlides,GuyonLekeufack}. See also \cref{sec:LOV}.

While \eqref{eq:introOSDE} is Markov on $\calM(\R^d)\times\R^d$, with $\calM(\R^d)$ denoting the set of finite measures on $\R^d$, direct simulation is difficult since 
$\calO_t$ is infinite-dimensional.
A practical numerical method must therefore  replace $\calO_t$ by a finite-dimensional representation, 
while controlling the induced error. This is the central objective of the present work. We here equip $\calM(\R^d)$ with a \emph{cylindrical} structure, i.e., taking a countable family of test functions $(f_k)_{k \in \mathbb N}$, we identify any finite measure $\mo \in \calM(\R^d)$ with a sequence $(\mo(f_k))_{k \in \mathbb N}$, where $\mo(f):=\int f(x) \, \mo(dx)$ denotes the integration of $f$ with respect to $\mo$. For any truncation level $K$, we then heuristically approximate $\calO_t$ in \eqref{eq:introOSDE} by a finite dimensional process $Z_t^K:=(\calO_t(f_1), \dotso, \calO_t(f_K))$. In Section 3, we introduce approximations $(X^K_t, Z_t^K)$ as solutions to finite dimensional stochastic differential equations \eqref{eq:CSDE}, which we call  cylindrical projections of the occupied diffusion \eqref{eq:introOSDE}. In Theorem~\ref{thm:convergence}, under standard Lipschitz assumptions on $(\lambda, b, \sigma)$, we prove that for any fixed horizon $T \in (0,\infty)$, 
\begin{align*}
    \mathbb E \left[ \sup_{0 \leq t\leq T} |X_t-X^K_t|^2 \right]^{1/2} \leq \frac{C}{K}, \quad \text{for some positive constant $C$.}
\end{align*}
By embedding $(Z_t^K)$ to the space of positive measures, we obtain the same convergence rate from $(Z_t^K)$ to $(\calO_t)$ with respect to a cylindrical norm similar to \cite{SonerTissotZhang}.

In Section 4, we implement an Euler--Maruyama scheme for the cylindrical projections and investigate several examples:
Cranston--Le Jan and Raimond self-attracting diffusions \cite{CranstonLeJan,Raimond}, as well as a variant of the 
local occupied volatility (LOV) model \cite{TissotLOV}.
These experiments illustrate the quantitative convergence of cylindrical projections. 
Motivated by Monte Carlo methods and derivatives pricing,  we include in Section 5 a weak error analysis of the proposed projection of OSDEs. 

From a computational perspective, cylindrical projections provide a general alternative
to both direct-history simulation and spectral moment approximations. A direct
Euler--Maruyama discretization of an occupied diffusion replaces \(\mathcal{O}_t\) by the empirical
occupation measure of the whole past trajectory; for self-interacting convolution kernels
this leads to \(O(N^2)\) kernel evaluations over \(N\) time steps and requires caching the entire discrete path in active memory (RAM). 
Fourier-moment methods can reduce this cost to \(O(N M_L)\), where \(M_L\) is
the number of frequency modes, but they are mainly suited to smooth translation-invariant
convolution kernels and require a periodic or truncated-domain spectral representation.
By contrast, our cylindrical approximation compresses the occupation flow itself into a
finite-dimensional Markovian state. It avoids the growth of the state dimension with the
time grid, applies to general coefficients \((\lambda,b,\sigma)\) depending on the
occupation measure, including nonlinear and non-convolutional dependencies, and yields
a deterministic projection error controlled by the strong and weak convergence results.

\paragraph{Related Literature. }   
The modeling of path-dependent systems draws from several intersecting fields. First, the class of SDEs considered here generalizes the so-called self-interacting  diffusions, introduced in  \citet{CranstonLeJan,Raimond}, and  further developed in \cite{Benaim1, Benaim2, Benaim3, Benaim4}. Let us also mention the recent works of  \citet{DuJiangLi} and \citet{ChassagneuxPages},  linking  ergodic dynamics of McKean-Vlasov type with  a specific class of OSDEs. 

In finance, OSDEs provide a subclass of  path-dependent volatility models, see  \citet{Guyon2014, GuyonSlides, GuyonLekeufack}. The practical execution of these models relies on numerical integration techniques for SDEs, notably the classical work of \citet{KloedenPlaten} and \citet{GuyonHL}, and the recent quantization-based methods of \citet{GrasselliPages}. 

Cylindrical projections serve as a versatile tool on the space of measures and have been used across various contexts. Notably, \citet{GuoPhamWei} establish a general It\^o formula for flows of probability measures using cylindrical functions. Also,  
\citet{Larsson} employ a similar construction in the context of measure-valued martingales, while \citet{SonerTissotZhang} leverages this approach to prove a comparison result for  nonlinear occupied PDEs.

\paragraph{Outline. }  The remainder of this paper is structured as follows. \cref{sec:occDiff} introduces occupied diffusions,  their cylindrical projections, and corresponding simulation scheme. Our  convergence results are established in \cref{sec:convergence}, and illustrated in \cref{sec:numerical} through a series  of examples. We also compare cylindrical approximations with other numerical methods in \cref{sec:comparison}. \cref{sec:WeakError} finally provides a weak error analysis, and explores its implications for Monte Carlo methods and the pricing of financial derivatives. Appendix~\ref{app}  contains postponed proofs. 

\section{Occupied Diffusions}\label{sec:occDiff}

Let $d\ge 1$ and $\calM$ be the space of finite Borel measures in $\R^d$. Introduce  the dual pairing 
$$\mo(\phi) = \int_{\R^d}\phi(x) \mo(dx), \quad \mo\in \calM, \quad \phi:\R^d \to \R,$$ 
whenever the above  is well-defined. 
Write $B_r\subset \R^d$ for the open ball of radius $r$ centered at the origin and $B_r(x) = x+ B_r$. 
We also denote the $d\times d$ identity matrix by $I_d$. 

\begin{definition}
   Let $(g_j)_{j \in \mathbb N} \subset C_b^1(\mathbb R^d)$ be a separating family  such that $g_0$ is a constant function and 
$\sum_{j=0}^{\infty} \lVert g_j \rVert^2_{\calC^1} \leq 1. $ We then  
define the \textit{cylindrical norm} associated to $(g_j)$ by 
\begin{equation}\label{eq:cylindricalNorm}
    \lVert \mo \rVert:= \Big( \sum_{j=0}^\infty |\mo(g_j)|^2\Big)^{1/2}, \quad \mo \in \calM(\R^d).
\end{equation} Given $(\mo,x) \in \calM \times \R^d$, we also write  $\lVert(\mo,x) \rVert := (\lVert \mo \rVert^2+ |x|^2)^{1/2}$. 
\end{definition}

Fix a finite horizon $T>0$ and a stochastic basis $(\Omega, \calF, \F,\Q)$, where $\F = (\calF_t)_{t\in [0,T]}$ satisfies the usual conditions. 
Introduce the occupied stochastic differential equations (OSDE),
\begin{align}
\begin{cases}\label{eq:OSDE}
 d\calO_t= \delta_{X_t}\lambda(\calO_t,X_t) \, dt, & \quad \calO_0 = \mo\in \calM,\\[0.25em]
     \, dX_t = b(\calO_t,X_t) \, dt +   \sigma(\calO_t,X_t) \, dW_t,  & \quad X_0 = x\in \R^d,
  \end{cases}
 \end{align} 
 where $b:\calM\times \R^d\to \R^d$, $\sigma: \calM\times \R^d\to \R^{d\times d}$, and $W$ is  standard Brownian motion in $\R^d$. The map  $\lambda:\calM\times \R^d\to \R_+$ describes the rate of the stochastic clock $\Lambda_t := \int_0^t \lambda(\calO_t,X_t)dt$. In particular, setting $\lambda \equiv 1$ (i.e., $\Lambda_t = t$) leads to  the \textit{calendar time occupation flow},
 $$\calO_t(A) = \mo(A) + \int_0^t\delta_{X_s}(A)ds = \int_0^t\mathds{1}_{A}(X_s)ds, \quad A \in \calB(\R^d), \quad t\in [0,T].$$
 If $\lambda = \text{tr}(\sigma\sigma^\top)$, then $\Lambda_t = \text{tr}(\langle X \rangle_t)$ is the (aggregate) quadratic variation clock.  Under the following assumption, \eqref{eq:OSDE} admits a unique strong solution $(\calO,X)$; see \cite[Lemma 3.2]{SonerTissotZhang}.
 \begin{assumption}\label{asm:growthLip}
     The coefficients $\lambda, b, \sigma$ satisfy the following linear growth and Lipschitz conditions: there exists $L> 0$ such that for all $\phi \in \{\lambda, b, \sigma\}$,
     \begin{align*}
        | \phi(\mo,x) | &\le L(1 + \lVert (\mo,x) \rVert ), \\[0.5em]
        | \phi(\mo,x) - \phi(\mo',x') | &\le L\lVert (\mo - \mo',x - x') \rVert.
     \end{align*}
     When $\phi = \sigma$, then $|\sigma(\mo,x)|$ denotes the Frobenius norm of $\sigma(\mo,x) \in \R^{d\times d}$. 
\end{assumption}

\subsection{Cylindrical Projection}\label{sec:cylinricalProj}

Similar to \citet{Larsson}, for any $K \in \mathbb N$, take an open cover $(U_k^K)_{\{k=1,\dotso, N_K\}}$ of $B_{R+1}$ such that for every $k\le N_K$,   $U_k^K \subset B_{1/K}(x_k^K)$ for some point $x_k^K \in B_{R+1}$. Set also $U_0^K=\R^d\setminus B_{R+1}$, $x_0^K \in \partial B_{R+1} $, 
and choose a partition of unity $(f^K_k)$ subordinate to $(U_k^K)_{\{k=0,\dotso,N_K\}}$, i.e., $$f_k^K \ge 0, \quad \; \sum_{k=0}^{N_K} f^K_k\equiv 1, \; \  \ \ \text{supp}(f^K_k) \subset U_k^K, \ \ k=0,\dotso, N_K, $$ 
with $\text{supp}(f) = \overline{\{x : \, f(x) > 0 \}}$. 
Introduce also the vector-valued function  $\bm{f}^K: \R^d\to \R^{N_K+1}$,  $\bm{f}^K(x) = (f^K_0(x), \dotso, f^K_{N_K}(x))$. 
The partition of unity is then used to project the occupation measure $\calO_t \in \calM$ onto the finite-dimensional vector 
$$Z_t^K  := \calO_t(\bm{f}^K) = (\calO_t(f_0),\ldots, \calO_t(f_{N_K}))\in \R^{N_K+1}.$$ Conversely, given $z \in \R^{N_K+1}$, an occupation measure is reconstructed   through the lift
\begin{align*}
    \mo^K: \R^{N_K+1} \to \calM,  
    \quad  z \mapsto \mo^K(z) = \sum_{k=0}^{N_K} z_k^K \delta_{x_k^K} \in  \calM.
\end{align*}
For $\phi:\calM\times \R^d\to \R$, 
define the projected map  $\phi^K(z,x) = \phi(\mo^K(z),x)$, $(z,x) \in  \mathbb \R^{N_K+1} \times \mathbb R^d$. 
We then approximate OSDE \eqref{eq:OSDE} by the \emph{projected OSDE}, 
\begin{align}\label{eq:CSDE}
\begin{cases}
    dZ_t^K =\bm{f}^K(X_t^K) \lambda^K(Z_t^K, X_t^K) \, dt, & \quad z_0= \mo(\bm{f}^K), \\[0.5em]
    dX_t^K= b^K(Z_t^K,X_t^K) \, dt + \sigma^K(Z_t^K,X_t^K) \, dW_t, & \quad X^K_0=x \in \mathbb R^d,
\end{cases}
\end{align}
with initial conditions  $x,\mo$ as in \eqref{eq:OSDE}. Since $(\calO,X)$, solution to \eqref{eq:OSDE}, is Markov,  the system \eqref{eq:CSDE} is Markov as well.

\subsection{Simulation Scheme} \label{sec:simulation}
The simulation of trajectories from the projected SDE \eqref{eq:CSDE} is straightforward and summarized in \cref{alg:simOSDE}. 
The Euler-Maruyama scheme is used  in Step III. to update the process. Alternatively, exact or higher order schemes could be used depending on $\lambda,b,\sigma$; see \cite[Chapter 10]{KloedenPlaten}.  

In Step III 1. of \cref{alg:simOSDE}, the occupation vector $Z^K_{\text{\tiny tmp}}$ is a temporary variable reused at each iteration. Indeed, as the joint process  $(Z^K,X^K)$ is Markov, updating $X$ only requires the current occupation times rather than the entire history. This provides a significant computational and storage advantage over general path-dependent SDEs, where the state-space dimension typically grows with each time step.  See Section~\ref{sec:comparison} for a performance comparison with alternative schemes. 

\begin{algorithm}[H]
\caption{\textbf{Euler-Maruyama Scheme,  Projected OSDEs (one sample)} }\label{alg:simOSDE}

\vspace{1mm}
 \textbf{Given:} $x_0\in \R^d$, $T>0$, $N \in \N$,    $\{f_k\}_{k\in \N}$ (separating family), $K \in \N$ (truncation level). 
 \\[-0.5em] 
\noindent\rule{\textwidth}{0.5pt} 
\begin{itemize}
 \setlength \itemsep{0.5em}

\item[I.] Initialize $X_0^K = x_0$ and  temporary variable $Z^K_{\text{\tiny tmp}} = 0\in \R^{N_K+1}$. 

\item[II.] \textbf{Generate} independent Brownian increments $\delta W_n = W_{t_{n+1}} - W_{t_{n}} \in \R^d$, $n<N$. 

\item[III.] \textbf{For} $n=0,\ldots,N-1$:  
\begin{enumerate}
 \setlength \itemsep{1.25em}
     \item 
     $ Z^K_{\text{\tiny tmp}} \ \longleftarrow \ Z^K_{\text{\tiny tmp}} + \lambda^K( Z^K_{\text{\tiny tmp}},X_{t_n}^{K}) \bm{f}^K(X_{t_n}^{K}) \delta t$,
     
     \item 
$X_{t_{n+1}}^{K} \ = \  X_{t_{n}}^{K}  + b^K(Z^K_{\text{\tiny tmp}},X_{t_{n}}^{K}) \delta t + 
    \sigma^K(Z^K_{\text{\tiny tmp}},X_{t_{n}}^K)\delta W_{n}$.

 \end{enumerate}

 \item[IV.] \textbf{Return} $X^K = (X_{t_0}^K, \ldots, X_{t_N}^K)$ and $Z^K_{\text{\tiny tmp}} = \calO_T(\bm{f}^K)$. 
   \end{itemize}
\end{algorithm}


\section{Main Results} \label{sec:convergence} 
For fixed horizon $T\in (0,\infty)$,  introduce the norms 
\begin{align}
            \lVert X\rVert_{\Q} &= \E^{\Q}\Big[ \sup_{0 \leq t \leq T} |X_t|^2\Big]^{1/2}, \label{eq:normX} \\[0.5em]
              \lVert (\calO,X)\rVert_{\Q} &=\E^{\Q}\Big[ \sup_{0 \leq t \leq T} \lVert(\calO_t,X_t) \rVert^2 \Big]^{1/2} = \E^{\Q}\Big[ \sup_{0 \leq t \leq T} \big( |X_t|^2 + \lVert \calO_t\rVert^2 \big)\Big]^{1/2}.  
\end{align}
and for any $R > 2$, define exit times for \eqref{eq:OSDE} and \eqref{eq:CSDE}, respectively, 
\begin{align}\label{eq:exitTime}
    \tau_{R}=\inf\{ t \geq 0: \lVert(\calO_t,X_t) \rVert \geq R \}, \quad 
    \tau_R^K=\inf\left\{ t \geq 0: \, \lVert (\mo(Z_t^K),X^K_t) \rVert \geq R  \right\}. 
\end{align}
We can now state the main results  of this section.

\begin{theorem}\label{thm:convergence}
Suppose that $\mo = \calO_0$ satisfies $\text{supp}(\mo) \subset B_R$ for some $R>2$. Then there exists $C>0$ that depends on the dimension $d$, time horizon $T$, radius $R$, and Lipschitz constants  of $b,\sigma, \lambda$ such that
    \begin{align}\label{eq:truncationerror}
        \lVert (\calO^{\tau_R},X^{\tau_R})-( \mo^K(Z^{K,\tau^K_R}) ,X^{K,\tau^K_R}) \rVert_{\Q} \leq \frac{C}{K}.
    \end{align}
\end{theorem}

\begin{proposition}\label{prop:exitbound}
Introduce the stopped process $(\calO^{\tau_R},X^{\tau_R}) = (\calO_{\cdot\wedge \tau_R},X_{\cdot \wedge \tau_R})$.
Then, 
\begin{align*}
    \lVert (\calO,X)-(\calO^{\tau_R},X^{\tau_R}) \rVert_{\Q}  \leq  \frac{C (1+\lVert(\mo,x) \rVert^2)^{0.675}}{R^{0.35}},
\end{align*}
where $C$ is a positive constant depending on $T$ and the Lipschitz constants of $b, \sigma,\lambda$.
\end{proposition}

\begin{remark}\label{rmk:truncation}
\begin{itemize}
    \item[(i)]  If we further impose in  Assumption~\ref{asm:growthLip} that $|\lambda(\mo,x)| \leq L(1+\lVert \mo \rVert)$ uniformly in $x$, then $\calO$ is bounded over $[0,T]$ by Gr\"{o}nwall inequality. Using the same arguments as in Theorem~\ref{thm:convergence}, one can prove that \eqref{eq:truncationerror} holds as well, 
where the constant $C$  depends on  $d$,  $T$, Lipschitz constants  of $b,\sigma, \lambda$, the $L^{\infty}$ norm of $\lambda$, but \emph{not} on the radius $R$. Choosing $R=K^{20/7}$ in Proposition~\ref{prop:exitbound} leads to 
\begin{align*}
    \lVert (\calO,X)-( \mo^K(Z^{K,\tau^K_R}) ,X^{K,\tau^K_R}) \rVert_{\Q} \leq \frac{C (1+\lVert(\mo,x) \rVert^2)^{0.675}}{K}. 
\end{align*}
\item[(ii)] Combining Proposition~\ref{prop:exitbound} with  Theorem~\ref{thm:convergence}, we get a convergence that depends on $K$ only. In \eqref{eq:convergencerate}, setting $R=\frac{1}{C} \log K$ yields 
\begin{align*}
    \lVert (\calO,X)-( \mo^K(Z^{K,\tau^K_R}) ,X^{K,\tau^K_R}) \rVert_{\Q} \leq \frac{C (1+|x|^2+\lVert \mo \rVert^2)^{0.675}}{(\log K)^{0.35}}. 
\end{align*}

\end{itemize}
\end{remark}

\subsection{Lenglart's Inequality Revisited}\label{sec:Lenglart} 

Toward the proof of \cref{prop:exitbound},  we establish a variant of Lenglart's inequality \cite{Lenglart1977}, which   provides probabilistic bounds on  the exit time $\tau_R$ defined in \eqref{eq:exitTime}. 

\begin{lemma}\label{lem:stopP}
  Under Assumption~\ref{asm:growthLip}, for any $\gamma \in (0,1)$, $R>2$,
    \begin{align}\label{eq:lenglart}
       \Q( \tau_R \leq T)  \leq \frac{ 4 e^{\iota \gamma T} (2-\gamma) \E^{\Q}[(1+\lVert(\mo ,x)\rVert|)^{3\gamma}]}{R^{3 \gamma}(1-\gamma) },
    \end{align}
    where $\iota$ is a constant that only depends on $b,\sigma$. 
\end{lemma}

\begin{proof}
    See Appendix~\ref{app:Lemma}.
\end{proof}

As can be seen in \eqref{eq:lenglart}, the bound decays polynomially in the radius $R$, while it grows exponentially with the time horizon $T$. Since both rates depend on $\gamma$, the upper bound may be optimized with respect to this parameter. 
In the proof of Proposition~\ref{prop:exitbound} below, we choose $\gamma = 0.9$, which yields a sufficiently tight upper bound for our purposes.


\begin{proof}[Proof of \cref{prop:exitbound}]
Let $\bar X=X-X^{\tau_R}$ and $\bar \calO=\calO-\calO^{\tau_R}$. In view of the definitions of $(\calO^{\tau_R}, X^{\tau_R})$ and $\lVert \cdot \rVert_{\Q} $, we have  
\begin{align*}
      \lVert (\calO,X)-(\calO^{\tau_R}, X^{\tau_R}) \rVert^2_{\Q} = \E^{\Q} \left[ \sup_{\tau_R \leq t \leq T} \lVert(\bar \calO_t,\bar X_t) \rVert^2  \right]. 
\end{align*}
Using the strong Markov property of \eqref{eq:OSDE} and conditioning on $\tau_R=t_0$, $(\calO_{t_0},X_{t_0})=(x,\mo)$ with $\lVert (x,\mo) \rVert^2=R^2-1$, then $(\calO_t,X_t)_{t\ge t_0}$  solves the OSDE
\begin{align*}
    \begin{cases}
        \calO_t=\mo+ \int_{t_0}^t \delta_{X_s} \lambda(\calO_s,X_s) \,ds,  \\[0.25em]
        X_t=x+\int^t_{t_0} b(\calO_s, X_s ) \,dt + \int_{t_0}^t \sigma (\calO_s, X_s) \, dW_s.
    \end{cases}
\end{align*}
It is then straightforward to show that 
\begin{align*}
    |\bar X_t|^2 \leq 2 (t-t_0) \int_{t_0}^t b^2(\bar \calO_s+\mo, \bar X_s+x) \,ds +2 |M_t|^2, 
\end{align*}
with the martingale  $M_t:=\int_{t_0}^t \sigma (\bar \mo + \calO_s, x + \bar X_s) \, dW_s$. From the linear growth of $b$,
\begin{align*}
    |\bar X_t|^2 \leq 2 (t-t_0) L^2 \int_{t_0}^t \left ( R^2+ \sup_{t_0 \leq u \leq s} \lVert(\bar \calO_u,\bar X_u) \rVert^2 \right) \,ds +2 |M_t|^2.
\end{align*}
Taking supremum over $t$ on both sides and applying Doob's martingale inequality, 
\begin{align*}
    \E^{\Q} \left[\sup_{t_0 \leq s \leq t} |\bar X_s|^2 \right] &\leq  [2 (t-t_0) + 8] L^2 \int_{t_0}^t \left(R^2 + \E^{\Q}\left[\sup_{t_0 \leq u \leq s} \lVert(\bar \calO_u,\bar X_u) \rVert^2 \right]  \right) ds.
\end{align*}
Using the definition of cylindrical norm, 
\begin{align*}
       \lVert \bar \calO_s \rVert^2 &=   \sum_{j=0}^{\infty} \left(\int_{t_0}^s g_j(X_u) \lambda (\calO_u,X_u) \, du\right)^2 \\[0.5em]
       &\leq  \sum_{j=0}^{\infty} \left(\int_{t_0}^s |g_j(X_u)|^2 \, du \right) \left(\int_{t_0}^s  \lambda(\calO_u, X_u)^2 \, du \right) \\[0.5em]
      &\leq    \left(\int_{t_0}^s \sum_{j=0}^{\infty} |g_j(X_u)|^2 \, du \right) \left(\int_{t_0}^s  \lambda(\calO_u, X_u)^2 \, du \right) \\[0.5em]
       &= (s-t_0) \int_{t_0}^s  \lambda(\calO_u, X_u)^2 \, du, 
\end{align*}
recalling that $\sum_{j=0}^{\infty} \lVert g_j \rVert^2_{\calC^1} \leq 1$,
and hence
\begin{align*}
    \E^{\Q} \left[\sup_{t_0 \leq s \leq t} \lVert \bar \calO_s \rVert^2 \right]  \leq (t-t_0) L^2 \int_{t_0}^t \left(R^2 + \E^{\Q}\left[\sup_{t_0 \leq u \leq s}  \lVert(\bar \calO_u,\bar X_u) \rVert^2 \right]  \right) ds.
\end{align*}
We therefore obtain the inequality 
\begin{align*}
    \E^{\Q} \left[\sup_{t_0 \leq s \leq t} \left( |\bar X_s|^2+\lVert \bar \calO_s \rVert^2 \right)  \right] \leq L^2 (3(t-t_0)+8)\int_{t_0}^t \left(R^2 + \E^{\Q}\left[\sup_{t_0 \leq u \leq s}  \lVert(\bar \calO_u,\bar X_u) \rVert^2 \right]  \right) ds.
\end{align*}
Hence by Gr\"{o}nwall inequality, with some positive constant $C$ depending on $L$ and $T$
\begin{align*}
     \E^{\Q} \left[\sup_{t_0 \leq s \leq t} \left( |\bar X_s|^2+\lVert \bar \calO_s \rVert^2 \right)  \right] &\leq  8L^2 \left((t-t_0)^2+(t-t_0)\right) R^2 \exp\left(8L^2 \left((t-t_0)^2+(t-t_0)\right) \right) \\
     &\leq C \left((t-t_0)^2+(t-t_0)\right) R^2. 
\end{align*}
Thus, we obtain that 
\begin{align*}
  \lVert (X,\calO)-(\calO^{\tau_R},X^{\tau_R}) \rVert_{\Q}^2 \leq CR^2 \E^{\Q} \left[ ((T-\tau_R)^+)^2 +(T-\tau_R)^+ \right].
\end{align*}
In view of  Lemma~\ref{lem:stopP}, 
\begin{align*}
\E^{\Q}[( (T-\tau_R)^+)^2+(T-\tau_R)^+] &=  \int_0^T (2u+1) \Q[\tau_R \leq T-u] \, du  \\[0.5em]
&\leq     \frac{ (T^2+1) e^{\iota \gamma T} (2-\gamma)  (1+|X_0|^2+\lVert \calO_0 \rVert^2)^{1.5\gamma}}{R^{3 \gamma}(1-\gamma) },
\end{align*}
and  setting $\gamma=0.9$ finally leads to 
$$  \lVert (X,\calO)-(\calO^{\tau_R},X^{\tau_R}) \rVert_{\Q}^2 \leq  \frac{C (1+|x|^2+\lVert \mo \rVert^2)^{1.35}}{R^{0.7}}.$$
\end{proof}

\begin{remark}
Using the same arguments as in the proof above, we also have 
\begin{align*}
    \lVert (\mo(Z^K),X^K)-(\mo(Z^{K,\tau^K_R}),X^{K,\tau^K_R}) \rVert_{\Q} \leq  \frac{C (1+\lVert (\mo,x) \rVert^2)^{0.675}}{R^{0.35}}.
\end{align*}
\end{remark}

\subsection{Proof of Theorem~\ref{thm:convergence}}\label{sec:proofMain}
We now prove our main result, which we decompose into the following steps. 
\begin{proof}
\underline{\textit{Step 1.}} Recall that $Z^K_0=\mo(\bm{f}^K)$. Let us estimate the difference of initial positions 
\begin{align*}
    \lVert \mo-\mo^K(Z^K_0) \rVert^2 &=\sum_{j=0}^{\infty} |\mo(g_j)- \mo^K(Z_0^K)(g_j)|^2 \\
    &=\sum_{j=0}^{\infty} \left|\mo(g_j)- \sum_{k=1}^{N_K} g_j(x^K_k) \mo(f^K_k) \right|^2,
\end{align*}
where we used that $\text{Supp}(\mo) \subset B_R$ and thus $\mo(f^K_0)=0$. Notice that 
\begin{align*}
    \Big |\mo(g_j)- \sum_{k=1}^{N_K} g_j(x^K_k) \mo(f^K_k) \Big|& \leq \Big|\sum_{k=1}^{N_K} \left(\int |g_j(x)-g_j(x^K_k)| f^K_k(x) \, d \mo(x) \right) \Big| \leq \frac{1}{K} ||g_j'||_{\infty}.
\end{align*}
Since $\sum_{j=0}^{\infty} ||g_j||_{\calC^1}^2 \leq 1$, we obtain that  $\lVert \mo-\mo^K(Z^K_0) \rVert^2 \leq 1/K^2$.

\bigskip
\noindent\underline{\textit{Step 2.}} For the simplicity of notation, in the rest we omit the stopping time $\tau_R,\tau_R^K$ in the superscript of $X^{\tau_R},\calO^{\tau_R},X^{K,\tau^K_R},  Z^{K,\tau^K_R} $.
 Let us denote $\theta_t=\calO_t- \mo^K(Z_t^{K})$, $\varepsilon_t=X_t-X^K_t$, and then 
\begin{align}
    \frac{d}{dt} \frac{\lVert \theta_t \rVert^2}{2}  &= \sum_{j=0}^{\infty}  \theta_t(g_j) \int g_j(x) \,\frac{\theta_t(dx)}{dt} \notag \\
    &= \sum_{j=0}^{\infty}  \theta_t(g_j)g_j(X_t)\lambda(\calO_t,X_t) \notag \\
    &- \sum_{j=0}^{\infty} \theta_t(g_j) \left( \sum_{k=0}^{N_K} g_j(x_k^K){f}_k^K(X_t^K) \right)\lambda(\mo^K(Z_t^K), X_t^K) \notag \\
    &= \sum_{j=0}^{\infty}  \theta_t(g_j)(g_j(X_t)-g_j(X_t^K))\lambda(\calO_t,X_t) \tag{I} \label{eq:I}\\
     &+ \sum_{j=0}^{\infty} \theta_t(g_j) \left(g_j(X_t^K)- \sum_{k=0}^{N_K} g_j(x_k^K){f}_k^K(X_t^K) \right)\lambda(\calO_t,X_t) \tag{II} \label{eq:II} \\
     &+  \sum_{j=0}^{\infty} \theta_t(g_j) \left( \sum_{k=0}^{N_K} g_j(x_k^K){f}_k^K(X_t^K) \right) (\lambda(\calO_t,X_k)- \lambda(\mo^K(Z_t^K), X_t^K)). \tag{III} \label{eq:III}
\end{align}
The term \eqref{eq:I} is bounded from above by 
\begin{align*}
  \sum_{j=0}^{\infty}  \left| \theta_t(g_j) \right| \lVert \nabla g_j \rVert_{\infty}|\varepsilon_t| |\lambda(\calO_t,X_t)| &\leq  |\varepsilon_t| |\lambda(\calO_t,X_t)| \sum_{j=0}^{\infty}  \left| \theta_t(g_j) \right| \lVert \nabla g_j \rVert_{\infty} \\
 &\leq |\varepsilon_t| |\lambda(\calO_t,X_t)| \lVert \theta_t \rVert .
\end{align*}
For any $ x \in B_R$, according to our construction of $f_k^K$, for any $x \in B_R$, $$\left| g_j(x)- \sum_{k=0}^{N_K} g_j(x_k^K)\bm{f}_k^K(x) \right|\leq \frac{\lVert \nabla g_j \rVert_{\infty}}{K}.$$ 
Keeping in mind that both $X_t$ and $X^{K}$ are stopped within the ball $B_R$,  the term of \eqref{eq:II} is bounded from above by 
\begin{align*}
     \sum_{j=0}^{\infty} |\theta_t(g_j)| \frac{\lVert \nabla g_j \rVert_{\infty}}{K}|\lambda(\calO_t, X_t)| \leq \frac{1}{K}|\lambda(\calO_t,X_t)| \lVert \theta_t\rVert .
\end{align*}
Using the Lipschitz property of $\lambda$, the term \eqref{eq:III} is bounded by
\begin{align*}
      L \sum_{j=0}^{\infty} |\theta_t(g_j)| \lVert g_j \rVert_{\infty}(\lVert \theta_t \rVert +|\varepsilon_t|) \leq L(\lVert \theta_t \rVert + |\varepsilon_t|) \lVert \theta_t \rVert.
\end{align*}
Together, we have the estimate 
\begin{align*}
    \frac{d}{dt}\lVert \theta_t \rVert^2 &\leq  2 | \lambda(\calO_t, X_t)| \left(\frac{1}{K}+|\varepsilon_t| \right) \lVert \theta_t \rVert + 2L(\lVert \theta_t \rVert +|\varepsilon_t|) \lVert \theta_t \rVert \\
    &\leq 2 LR \left(\frac{1}{K}+|\varepsilon_t| \right) \lVert \theta_t \rVert + 2L(\lVert \theta_t \rVert +|\varepsilon_t|) \lVert \theta_t \rVert,
\end{align*}
where we used the linear growth property of $\lambda$ and $\sqrt{1+|X_t|^2+||\calO||^2} \leq R$. From \textit{Step 1}, we have $||\theta_0||^2 \leq 1/K^2$. In view of   $\frac{2LR ||\theta_t||}{K}\leq LR||\theta_t||^2+\frac{LR}{K^2}$, we obtain that
\begin{align*}
 \lVert \theta_t \rVert^2  
  \leq  \frac{CR}{K^2} +CR  \int_0^t \lVert \theta_s \rVert^2+|\varepsilon_s|^2\, ds,
\end{align*}
and thus
\begin{align}\label{eq:thm1}
  \E^{\Q} \left[  \sup_{0 \leq s \leq t}\lVert \theta_s \rVert^2 \right] 
  \leq & \frac{CR}{K^2} +  CR\E^{\Q} \left[ \int_0^t  \sup_{0 \leq u \leq s} (\lVert \theta_u \rVert^2+ |\varepsilon_u|^2 )\, ds \right],
\end{align}
where $C$ is a positive constant only depending on the Lipschitz constants and $T$ that is allowed to change from line to line.

\bigskip

\noindent \underline{\textit{Step 3.}}  By definition of $\varepsilon_t$,
\begin{align*}
\frac{|\varepsilon_t|^2}{2} &\leq  t \int_0^t   \left(b(\calO_s,X_s)-b(\mo^K(Z^K_s), X^K_s)\right)^2dt   +   \left| \int_0^t \sigma(\calO_s,X_s)-\sigma(\mo^K(Z^K_s),X^K_s) \, dW_s \right|^2.
\end{align*}
Taking supremum over time, using Doob's martingale inequality and Lipschitz property of $b,\sigma$, we obtain that 
\begin{align*}
    \E^{\Q} \left[\sup_{0\leq s  \leq t} |\varepsilon_s |^2 \right] \leq C  \E^{\Q} \left[ \int_0^t  \sup_{0 \leq u \leq s} (\lVert \theta_u \rVert^2+ |\varepsilon_u|^2 )\, ds \right],
\end{align*}
Together with \eqref{eq:thm1}, by Gr\"{o}nwall inequality we conclude the result
\begin{align}\label{eq:convergencerate}
    \E^{\Q}\left[ \sup_{0 \leq s \leq T } (\varepsilon_s^2+ \lVert \theta_s \rVert_s^2 )\right] \leq \frac{CR}{K^2} \exp(CR). 
\end{align}
\end{proof}

\section{ Numerical Examples } 
\label{sec:numerical}

This section gathers examples of occupied SDEs, ranging from self-interacting diffusions to path-dependent volatility model in finance. We also verify numerically the  convergence rates established in \cref{thm:convergence}. 

\subsection{Cranston-Le Jan's Self-interacting  Diffusion} \label{sec:CranstonLeJan}
 
  Consider the  self-attracting diffusion introduced by 
    \citet{CranstonLeJan}, 
\begin{equation}\label{eq:SelfAttracting}
   dX_t = \beta\bigg(\int_0^t(X_s - X_t)ds\bigg) dt + dW_t, \quad X_0 =x, \quad \beta > 0. 
\end{equation}
Then \eqref{eq:SelfAttracting} corresponds to the occupied SDE, 
\begin{align}
\begin{cases}\label{eq:OSDEExplicit}
 d\calO_t= \delta_{X_t}\, dt, & \\[0.25em]
     \, dX_t = b(\calO_t,X_t) \, dt +  dW_t,  & b(\mo,x) = \beta \int_{\R}(y-x)\mo(dy).
  \end{cases}
 \end{align} 
Rewriting the drift in \eqref{eq:SelfAttracting} as $\beta t(\bar{X}_t - X_t)$, $\overline{X}_t = \frac{1}{t}\int_0^tX_s ds$, we gather that the process is attracted to its running average $\overline{X}$ when $\beta >0$,  and repelled from it otherwise; see \cref{fig:simulationsCranstonLeJan}. 
As shown in \cite{CranstonLeJan}, \eqref{eq:SelfAttracting} admits an explicit strong solution whose proof is reported here for completeness. 

\begin{figure}[H]
    \centering
     \caption{Sample paths $X$ and drift  from the Cranston-Le Jan diffusion \eqref{eq:SelfAttracting}. Dotted line is the running average of $X$ (i.e., the barycenter of $\calO$). }
   
    \label{fig:simulationsCranstonLeJan}

\begin{subfigure}[b]{0.49\textwidth}
        \centering
       \caption{$\beta = -5$.}
    \includegraphics[width=0.92\linewidth]{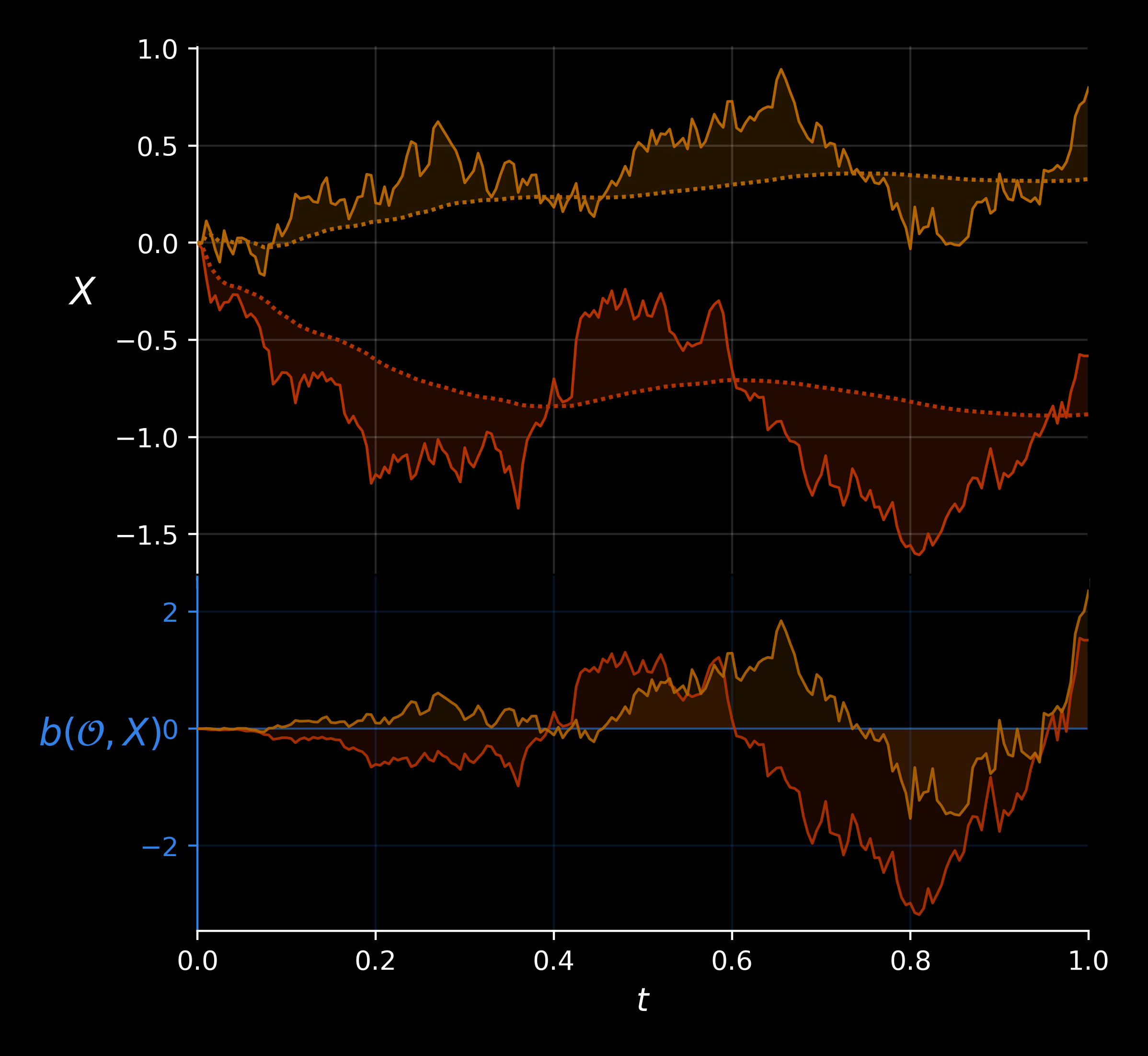}
    \label{fig:sim3}
\end{subfigure}
\begin{subfigure}[b]{0.49\textwidth}
 \centering
\caption{$\beta = 5$.}
    \includegraphics[width=0.9\linewidth]{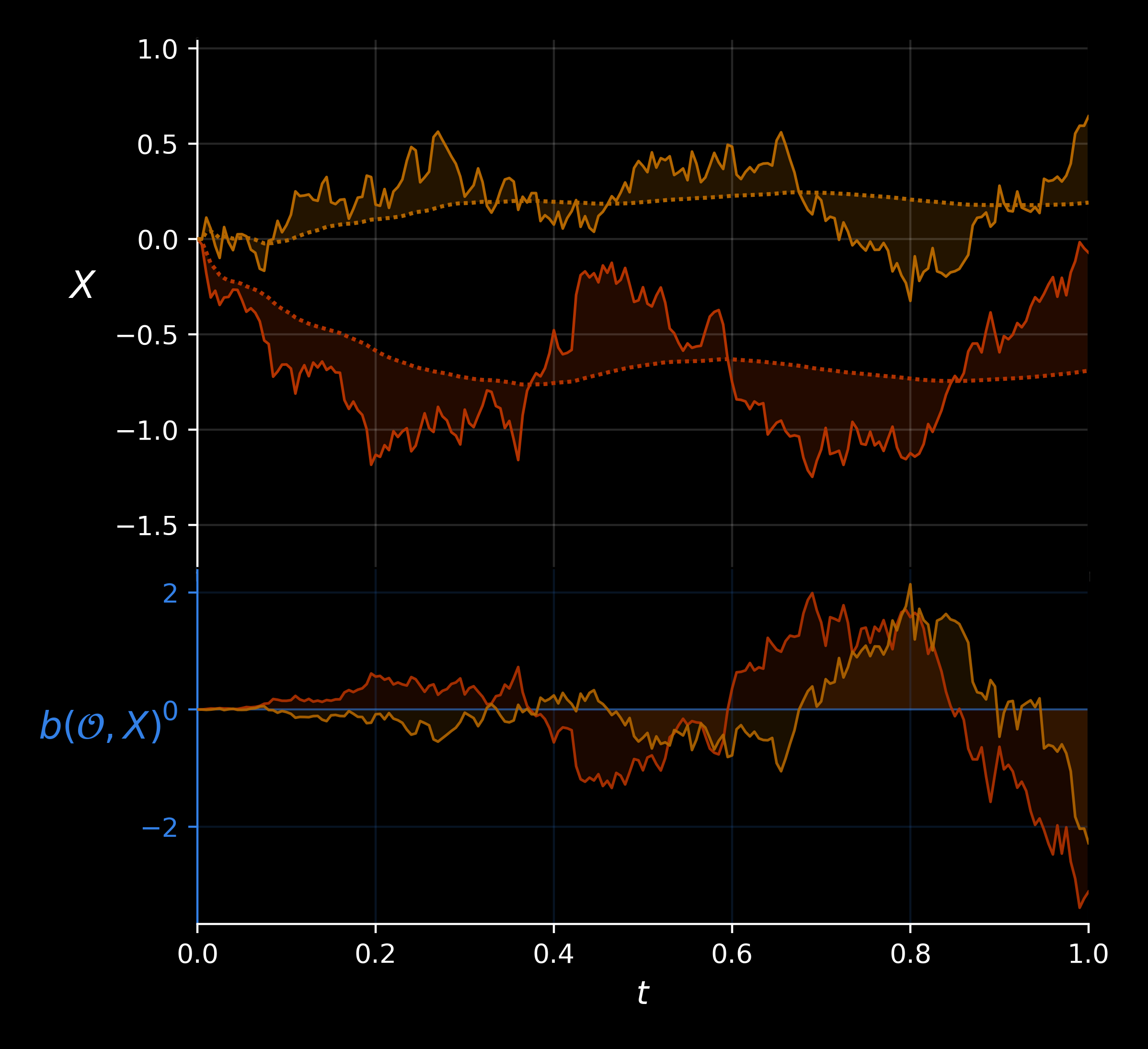}
    \label{fig:sim6}
\end{subfigure}
\end{figure}

\begin{proposition}\label{prop:cranstonExact}
  Cranston-Le Jan's diffusion \eqref{eq:SelfAttracting} admits a unique strong solution given by $X_t = x+ \int_0^t \kappa(t,s) dW_s,$ with the  Volterra kernel
  \begin{equation}\label{eq:Volterra}
      \kappa(t,s) = 1 -\beta s\int_s^t e^{\frac{\beta}{2}(s^2-u^2)}du, \quad t\ge s. 
  \end{equation}
\end{proposition}

\begin{proof} See Appendix~\ref{app:Volterra}. 
\end{proof}
Observe that  $X$ is a Gaussian Volterra process with constant mean function and covariance kernel $\C^{\Q}(X_t,X_s) = \int_0^{s\wedge t} \kappa(t,u)\kappa(s,u)du$. The Volterra and covariance kernels are displayed in \cref{fig:Kernels}. 
\begin{figure}
    \centering
     \caption{Volterra  and covariance kernel  of the Cranston-Le Jan diffusion \eqref{eq:SelfAttracting}, $\beta =1$.}
   
    \label{fig:Kernels}

\begin{subfigure}[b]{0.49\textwidth}
        \centering
       \caption{Volterra kernel $\kappa(t,s)$, $t\ge s$.} 
    \includegraphics[width=0.92\linewidth]{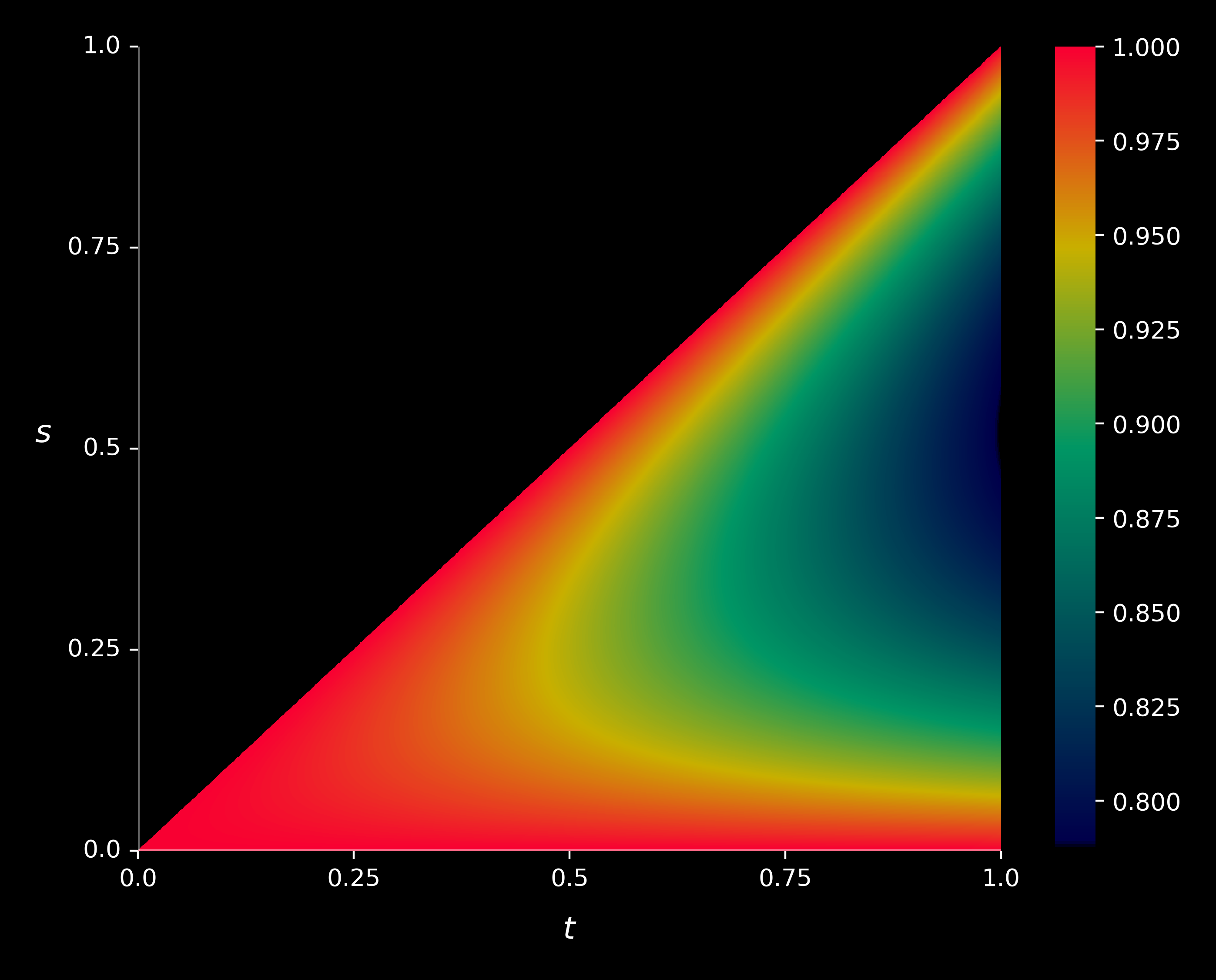}
    \label{fig:volterraKernel}
\end{subfigure}
\begin{subfigure}[b]{0.49\textwidth}
 \centering
 \caption{Covariance kernel $\int_0^{t\wedge s} \kappa(t,u)\kappa(s,u)du$.}
    \includegraphics[width=0.9\linewidth]{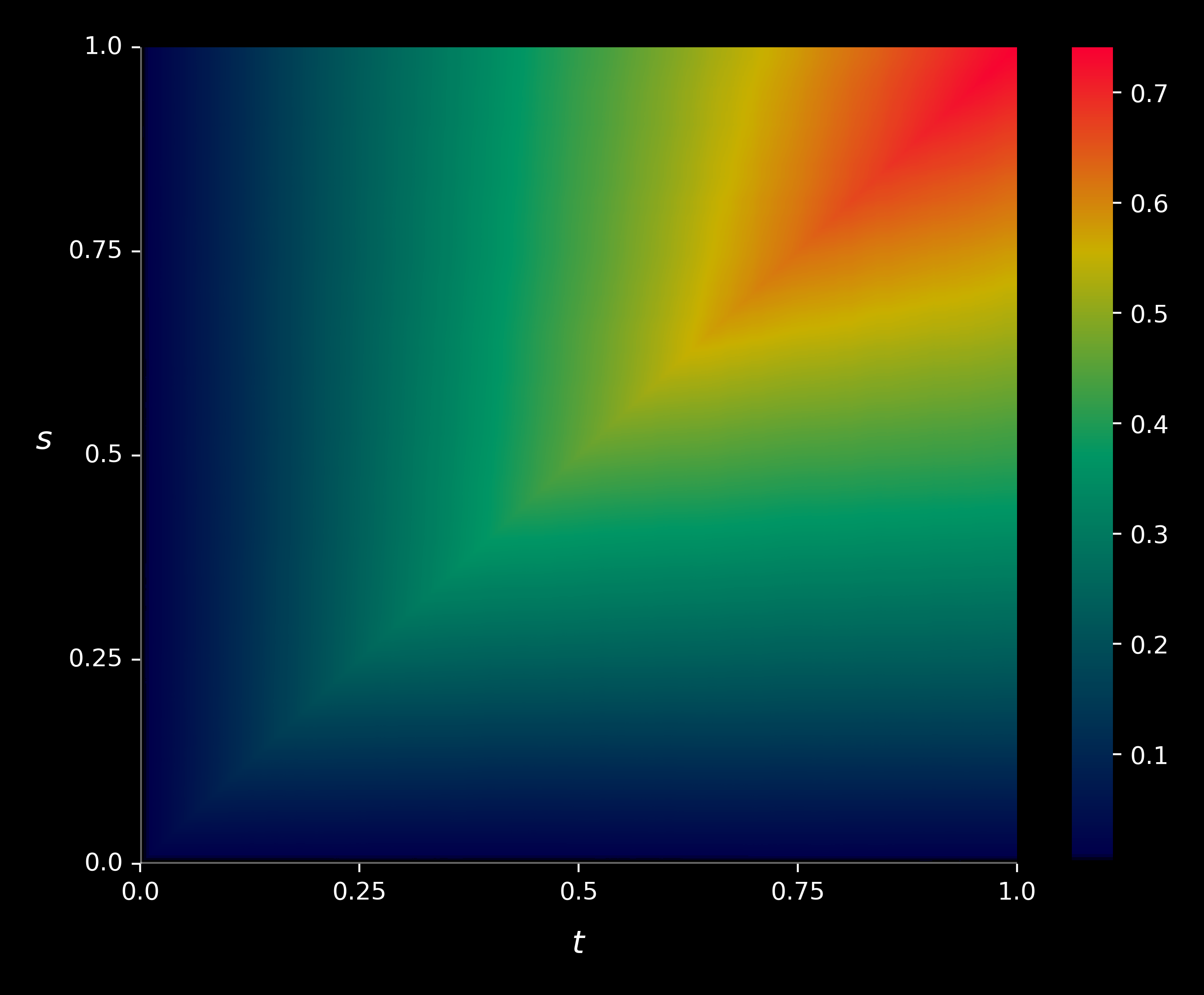}
    \label{fig:covarianceKernel}
\end{subfigure}
\end{figure}
Next, let us verify  \cref{thm:convergence} by  numerically comparing projected diffusions to the exact solution from \cref{prop:cranstonExact}. We use $\beta = 5$, $T=1$, $2^9$ time steps, and $2^{13}$ simulations, and choose $f_k^K = \mathds{1}_{A_k^K}$, where $(A_k^K)_{k=1}^K$ form a regular partition   of the compact interval $[-R,R]$, $R = 2$, and $A_0^K = \R \setminus [-R,R]$. The results are shown in \cref{fig:pathwiseCranstonLeJan,fig:pathwiseErrorCranstonLeJan,fig:convergenceCranstonLeJan}.   The error in \eqref{eq:truncationerror}  is split into its spatial and occupation components, both of which demonstrate a nearly linear convergence rate. 

\begin{figure}[H]
    \centering
    \caption{Sample path from  Cranston-Le Jan's diffusion ($\beta = 5$) for increasing values of $K$. }
   
    \label{fig:pathwiseCranstonLeJan}

\begin{subfigure}[b]{0.495\textwidth}
        \centering
       \caption{Full horizon $[0,1]$.}
    \includegraphics[width=0.92\linewidth,height = 2.2in]{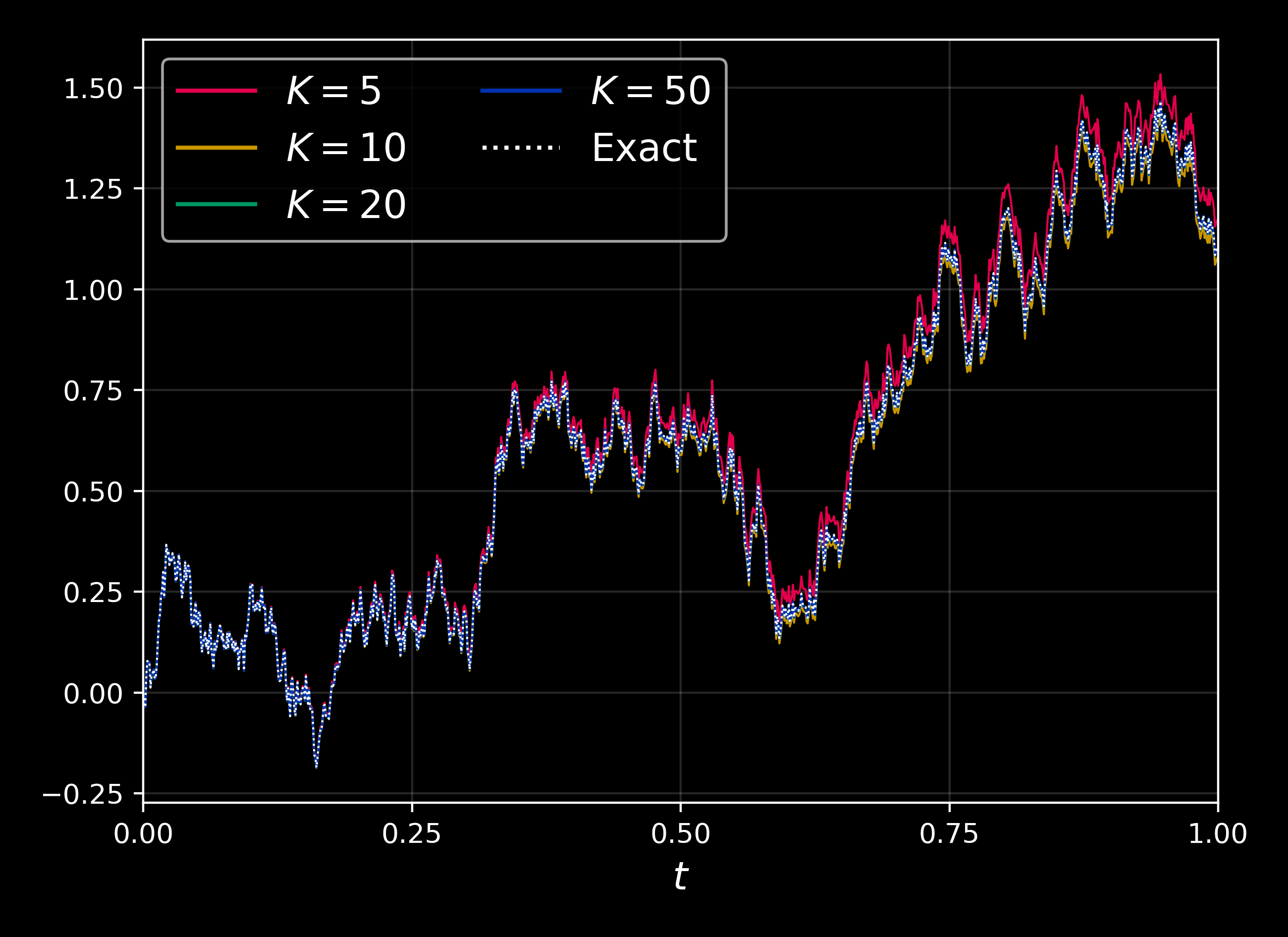}
\end{subfigure}
\begin{subfigure}[b]{0.495\textwidth}
 \centering
\caption{Zoom on $[0.75,1.]$.}
    \includegraphics[width=0.92\linewidth,height = 2.2in]{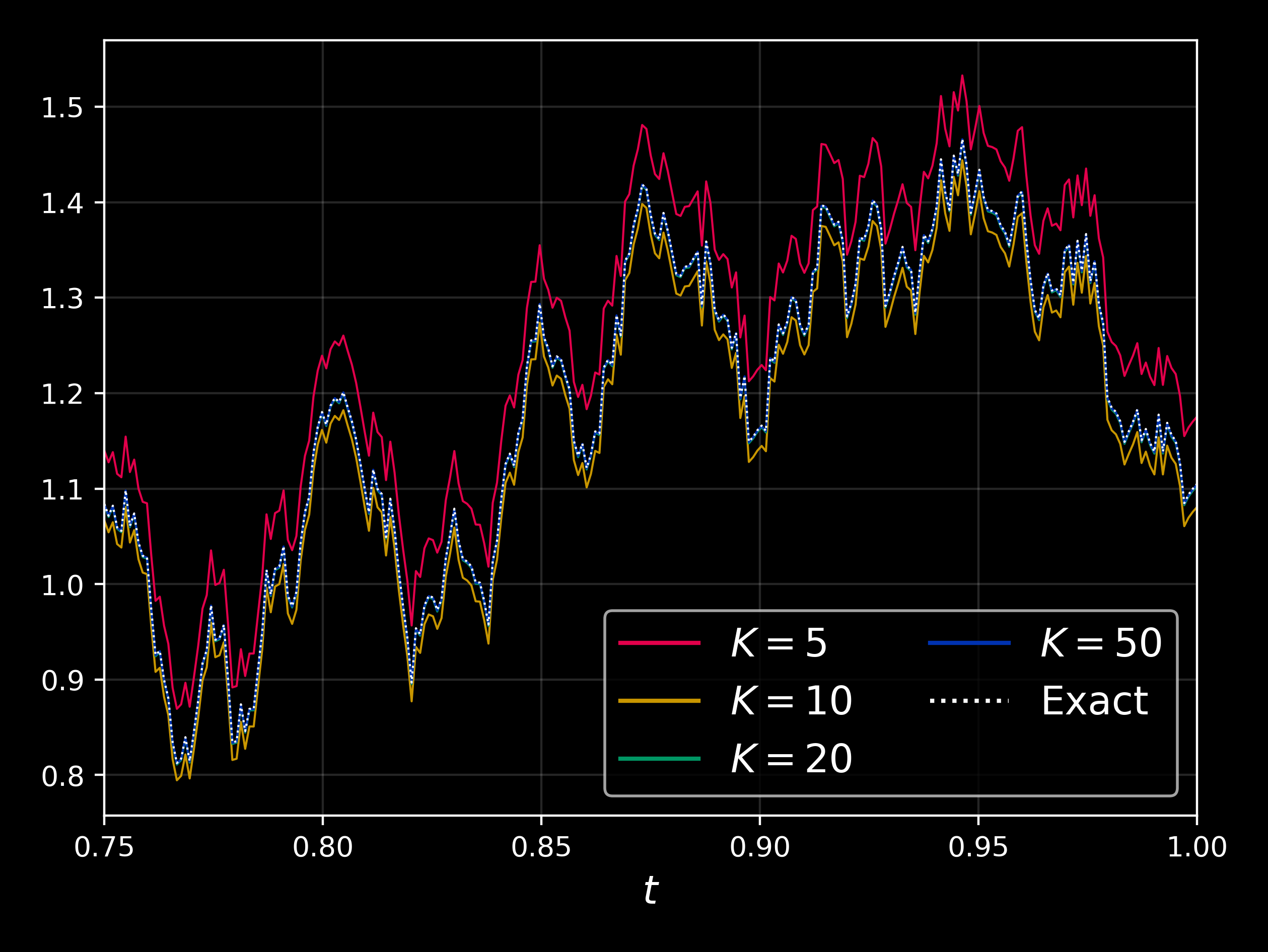}
\end{subfigure}
\end{figure}

\begin{figure}[H]
    \centering
    \caption{Pathwise errors $X-X^K$ ($\beta = 5$) for  increasing values of $K$. }
   
    \label{fig:pathwiseErrorCranstonLeJan}

    \includegraphics[width=0.5\linewidth,height = 2.1in]{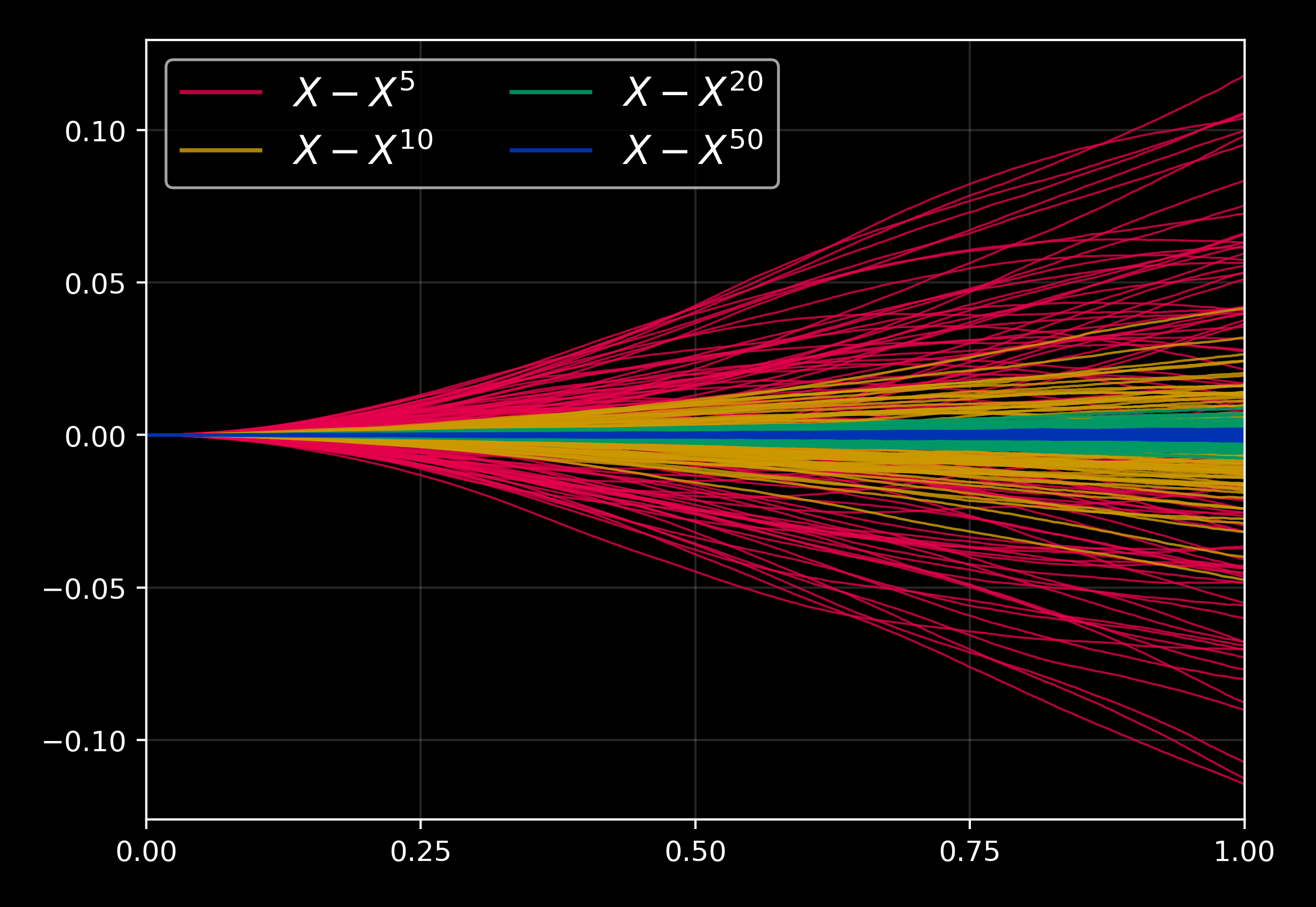}

\end{figure}

\begin{figure}[H]
    \centering
    \caption{Convergence rate of Cranston-Le Jan's diffusion \eqref{eq:SelfAttracting}, $\beta = 5$. }
   
    \label{fig:convergenceCranstonLeJan}

\begin{subfigure}[b]{0.49\textwidth}
        \centering
    \includegraphics[width=0.92\linewidth,height = 2.1in]{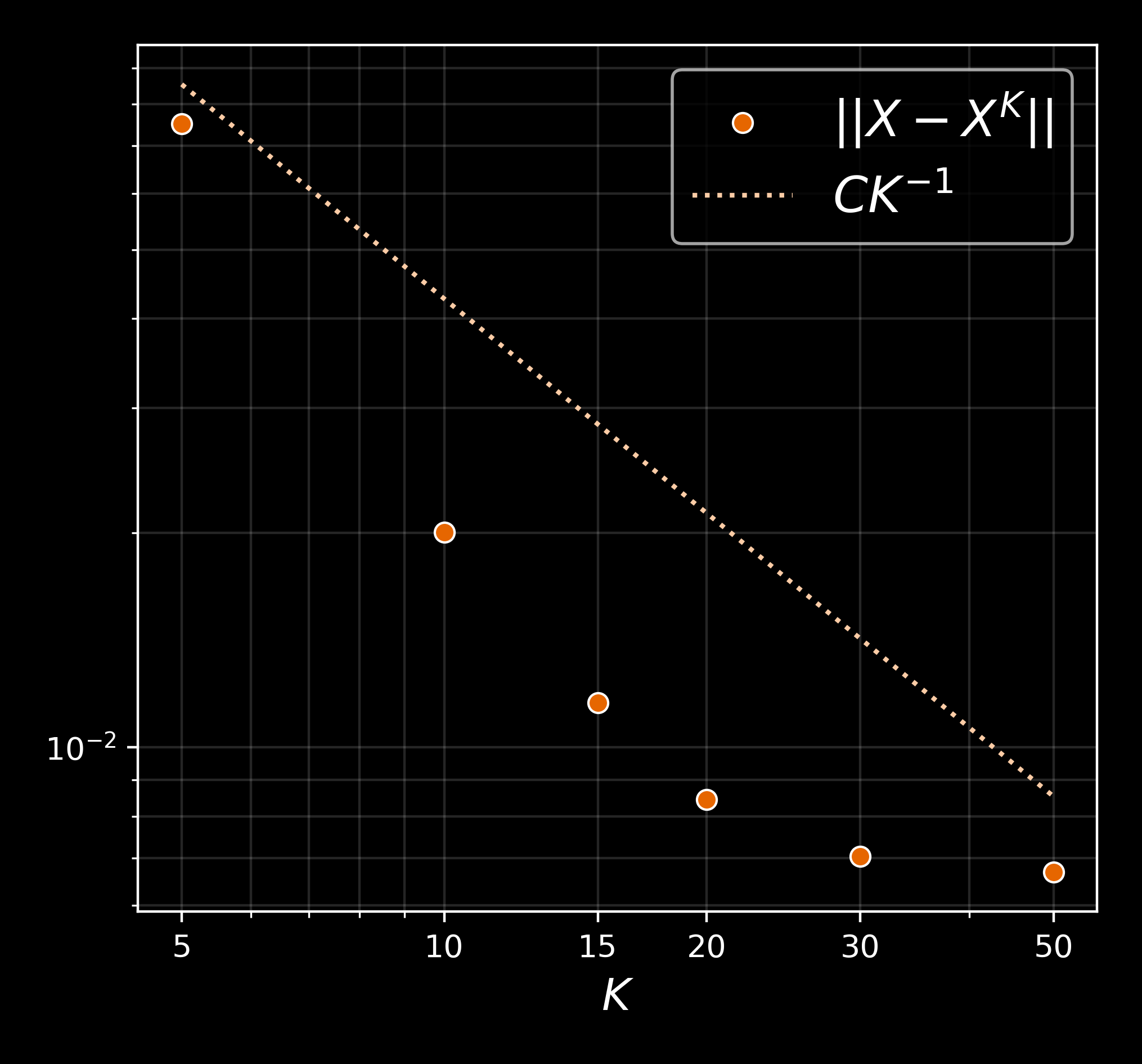}
\end{subfigure}
\begin{subfigure}[b]{0.49\textwidth}
 \centering
    \includegraphics[width=0.92\linewidth,height = 2.1in]{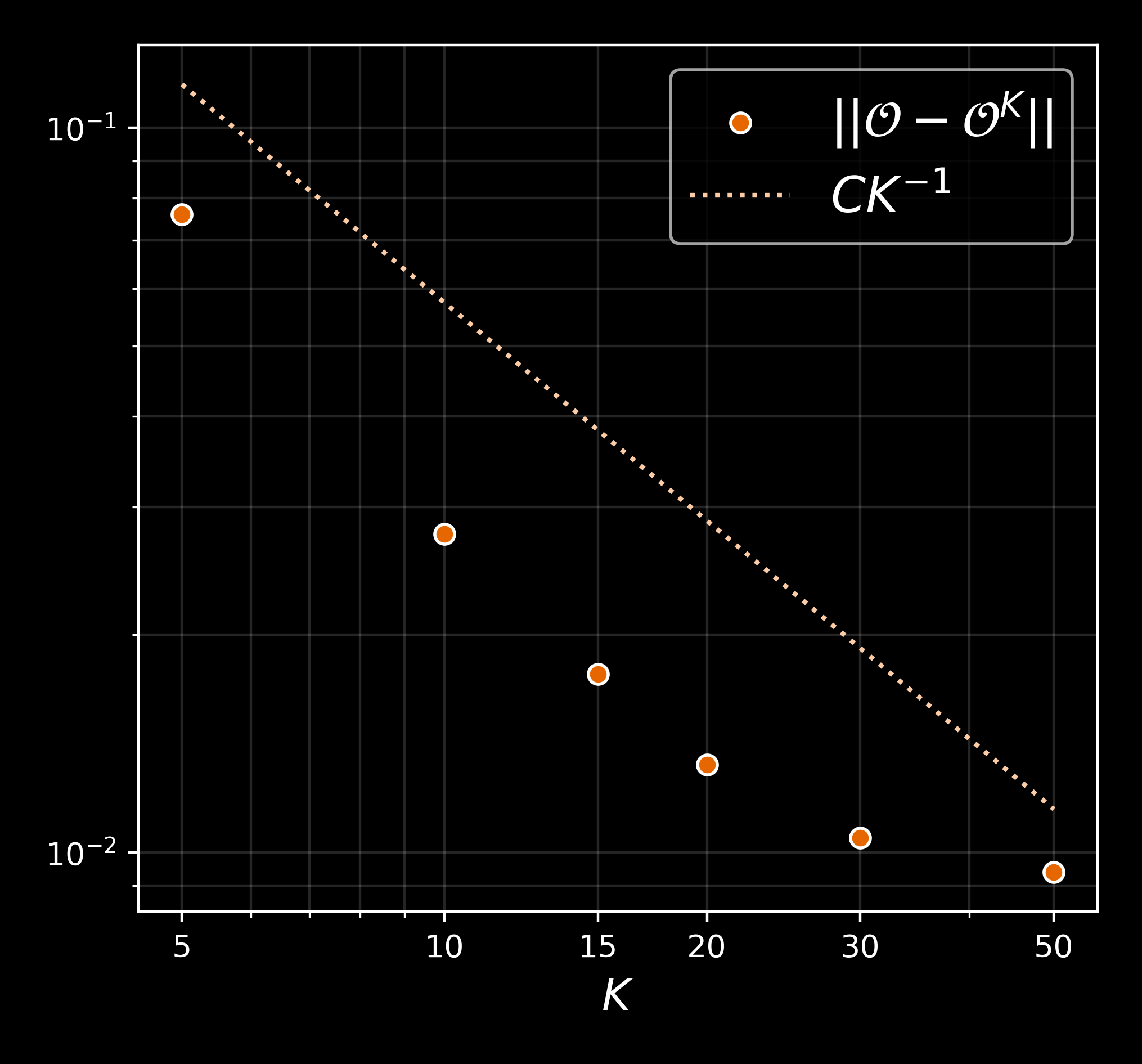}
\end{subfigure}
\end{figure}

\subsection{Raimond's Self-interacting Diffusion  }\label{sec:raimond}

Consider the following multi-dimensional 
diffusion introduced by \citet{Raimond}, 
\begin{equation}\label{eq:raimond}
    dX_t = \beta\bigg(\int_0^t \frac{X_s - X_t}{|X_s - X_t|}ds\bigg) dt + dW_t, \quad X_0 =x \in \R^d, 
\end{equation}
where $W$ is a Brownian motion in $\R^d$. Note that \eqref{eq:raimond} is self-attracting when $ \beta > 0$, and self-repelling otherwise. 
The associated  OSDE \eqref{eq:OSDE} is obtained by setting $\lambda \equiv 1$,  $\sigma \equiv I$, and 
\begin{equation}\label{eq:raimondDrift}
    b(\mo,x) = \beta \int_{\R^d} \varphi(y - x)\mo(dy),  \quad \varphi(x) = \frac{x}{|x|}.
\end{equation} 
Alternatively, note that 
$b(\mo,x) = - \nabla_x V(\mo,x)$
with the  potential $V(\mo,x) = \beta \int_{\R^d} |y-x| \mo(dy)$.  To remove  the singularity at $x = 0$ in the implementation, we replace the vector field $\varphi$  by $\varphi_{\varepsilon}(x) = \frac{x}{\sqrt{\varepsilon + |x|^2}}$ for small $\varepsilon>0$. It is then easily seen that 
the drift \eqref{eq:raimondDrift} satisfies  \cref{asm:growthLip}, with $\varphi_{\varepsilon}$ in lieu of $\varphi$. 

A sample trajectory of $X$ and its drift  are  depicted in \cref{fig:simulationsRaimond} for $d=2$ and  $\beta \in \{-1,5\}$. As expected, the drift increases in magnitude with $\beta$ (bottom chart), which intensifies the attraction of  the diffusion toward its own past trajectory. The cylindrical projection of $(\calO,X)$  is carried out using $f_k^K = \mathds{1}_{A_k^K}$, with boxes $A_k^K \subseteq \R^{2}$  forming a partition of $ [-R,R]^2$, $R = 2$ as shown in the top charts of \cref{fig:simulationsRaimond}. 

The convergence results are shown in \cref{fig:convergenceRaimond} with  $\beta = 5$, $T=1$, $2^9$ time steps, and $2^{13}$ sample paths.  Since no explicit solution of $(\calO,X)$ is known in this case, the cylindrical projections are compared to $ X^{\overline{K}}$ for large $\overline{K}$ (here $\overline{K} = 75$). Again, the truncation error decreases linearly, bounded only by sampling noise and discretization inaccuracies.



\begin{figure}[H]
    \centering
     \caption{Sample path (top) and drift  
     (bottom)  from Raimond's self-attracting diffusion \eqref{eq:raimond} for two values for $\beta$ (based on same Brownian  path). }
   
    \label{fig:simulationsRaimond}

\begin{subfigure}[b]{0.49\textwidth}
        \centering
       \caption{$\beta = -1$.}
    \includegraphics[width=0.96\linewidth,height = 2.1in]{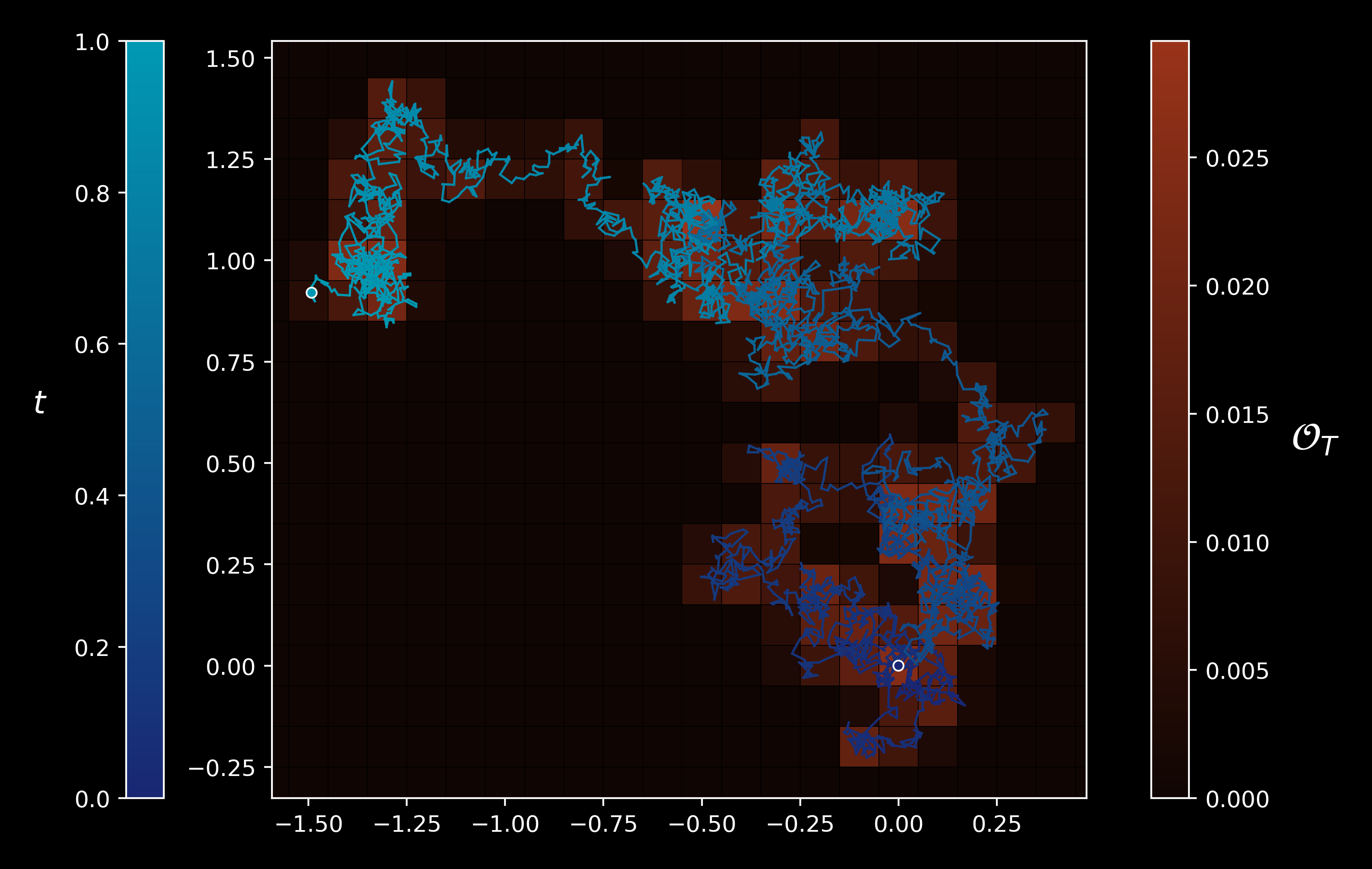}

     \includegraphics[width=0.96\linewidth]{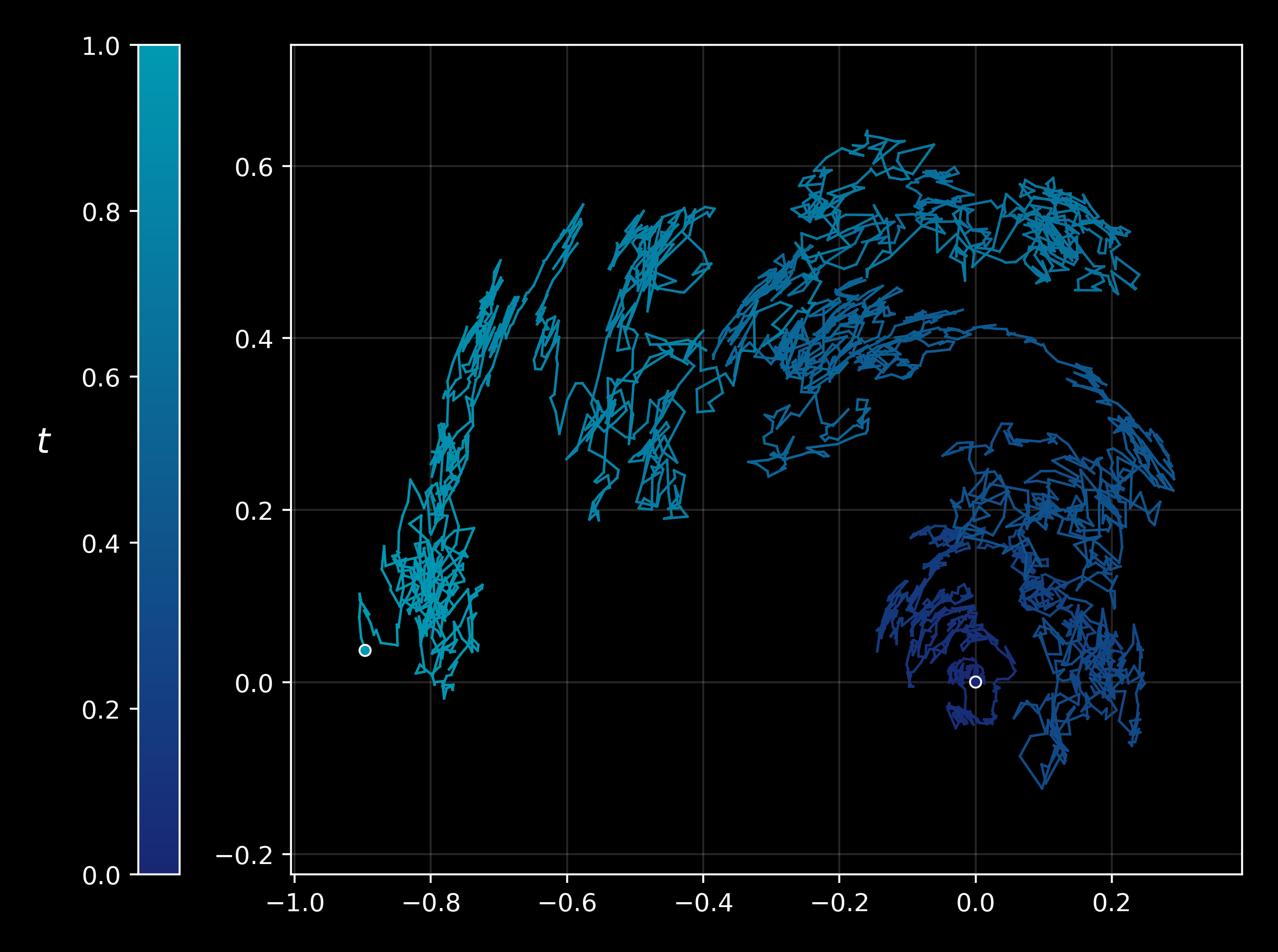}
\end{subfigure}
\begin{subfigure}[b]{0.49\textwidth}
 \centering
\caption{$\beta = 5$.}
    \includegraphics[width=0.96\linewidth,height = 2.1in]{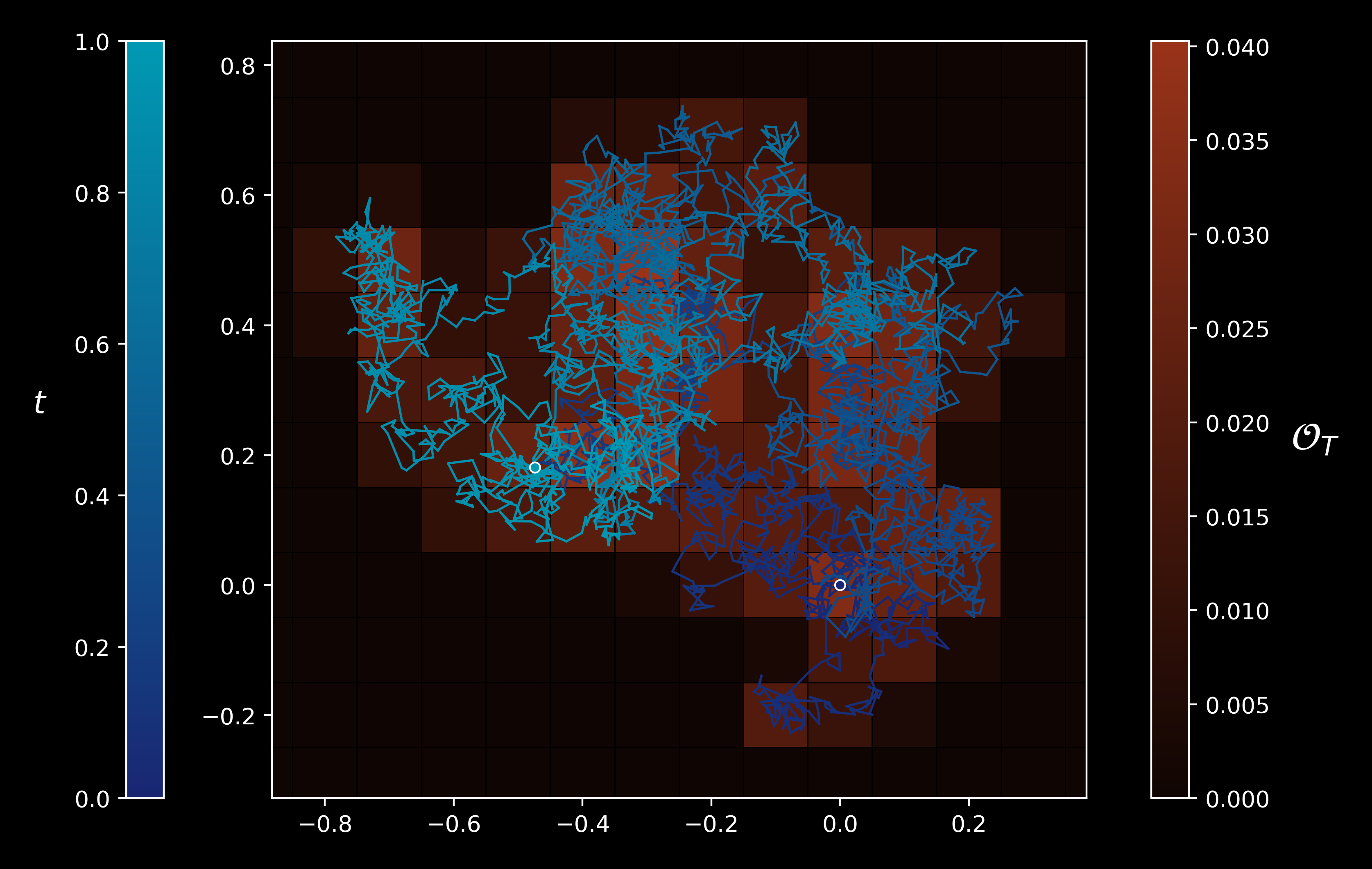}

     \includegraphics[width=0.96\linewidth]{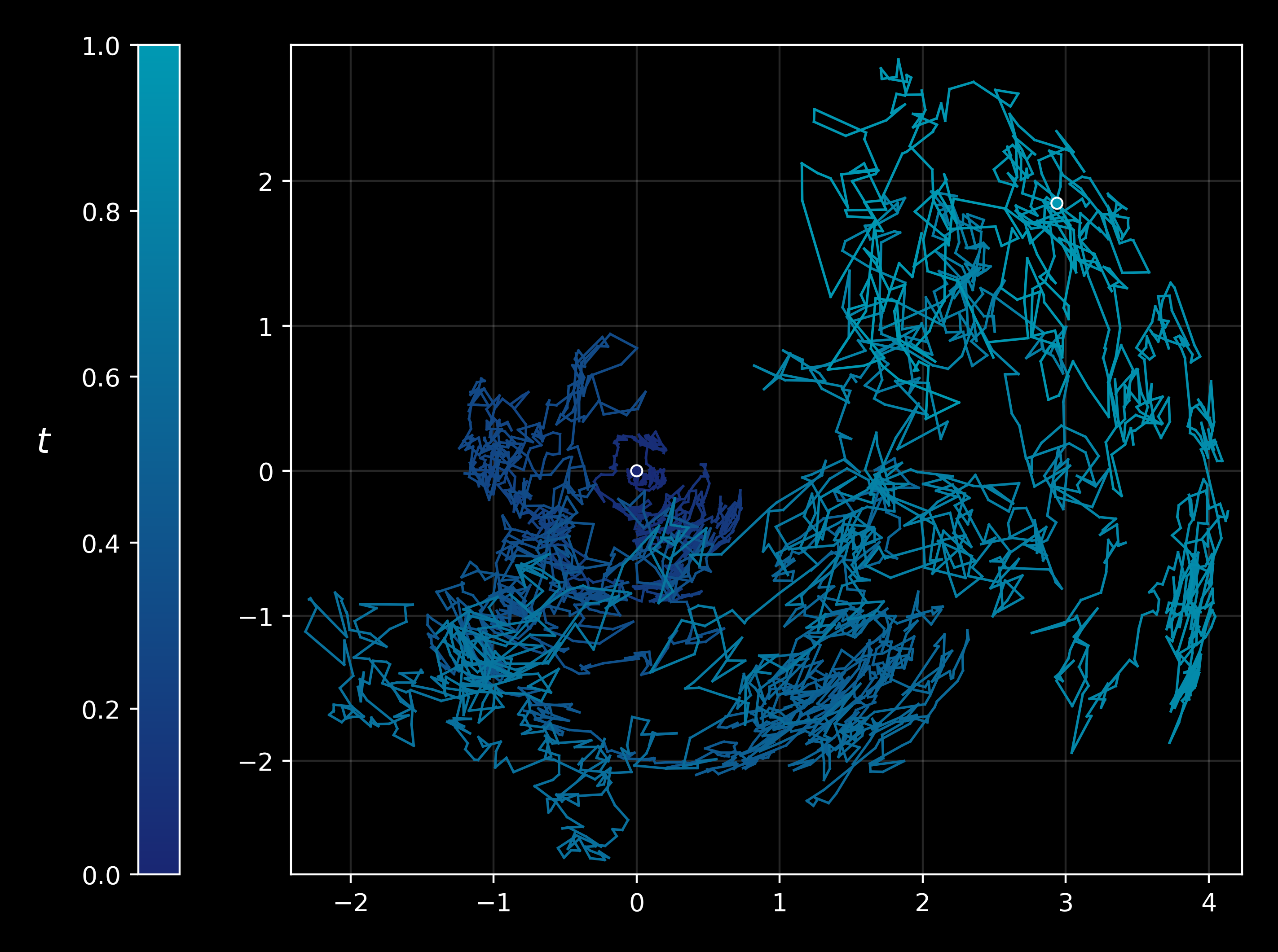}
\end{subfigure}
\end{figure}

\begin{figure}[H]
    \centering
    \caption{Convergence rate of Raimond's diffusion \eqref{eq:SelfAttracting}, $\beta = 5$. }
   
    \label{fig:convergenceRaimond}

\begin{subfigure}[b]{0.49\textwidth}
        \centering
    \includegraphics[width=0.92\linewidth,height = 2.1in]{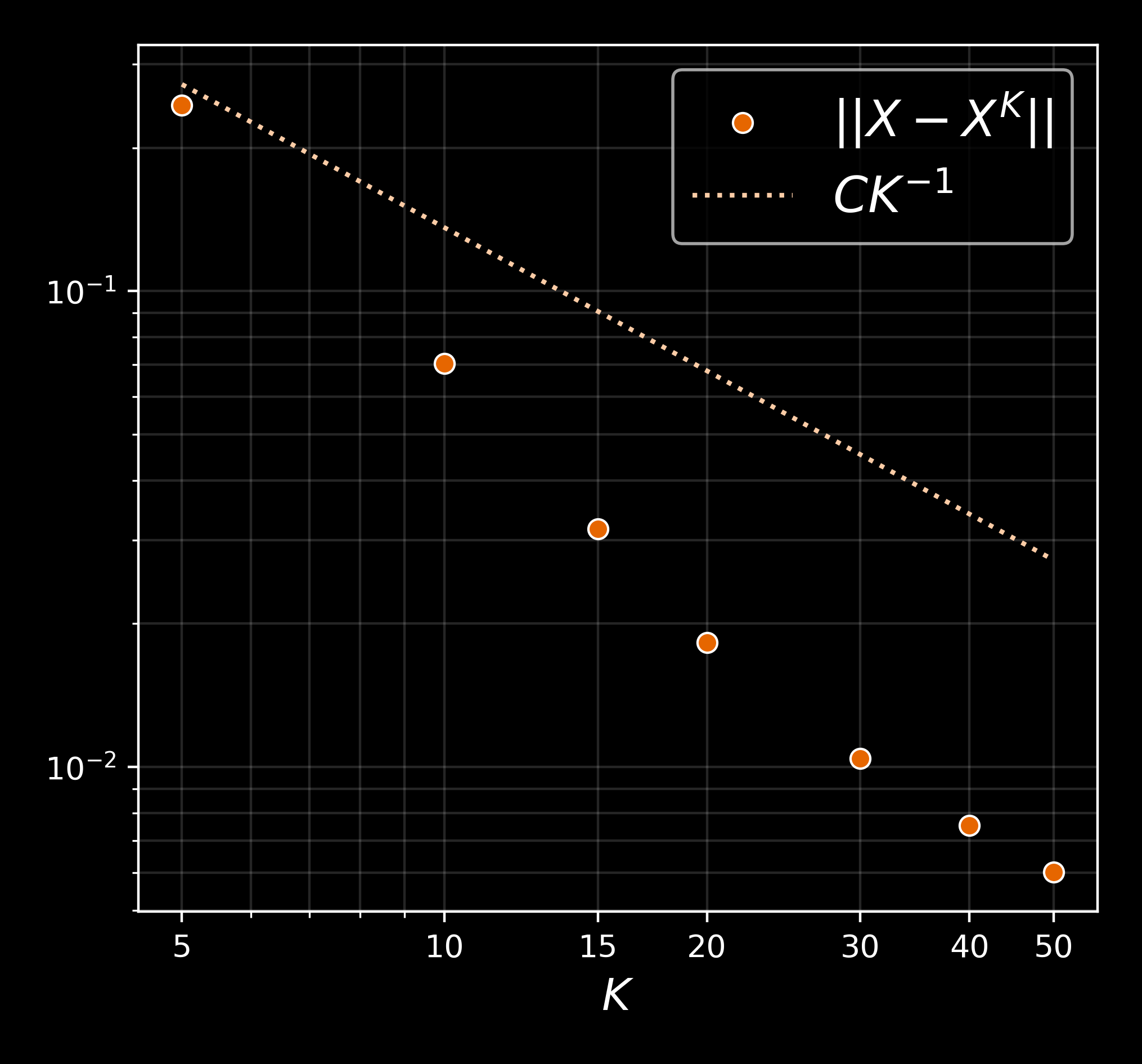}
\end{subfigure}
\begin{subfigure}[b]{0.49\textwidth}
 \centering
    \includegraphics[width=0.92\linewidth,height = 2.1in]{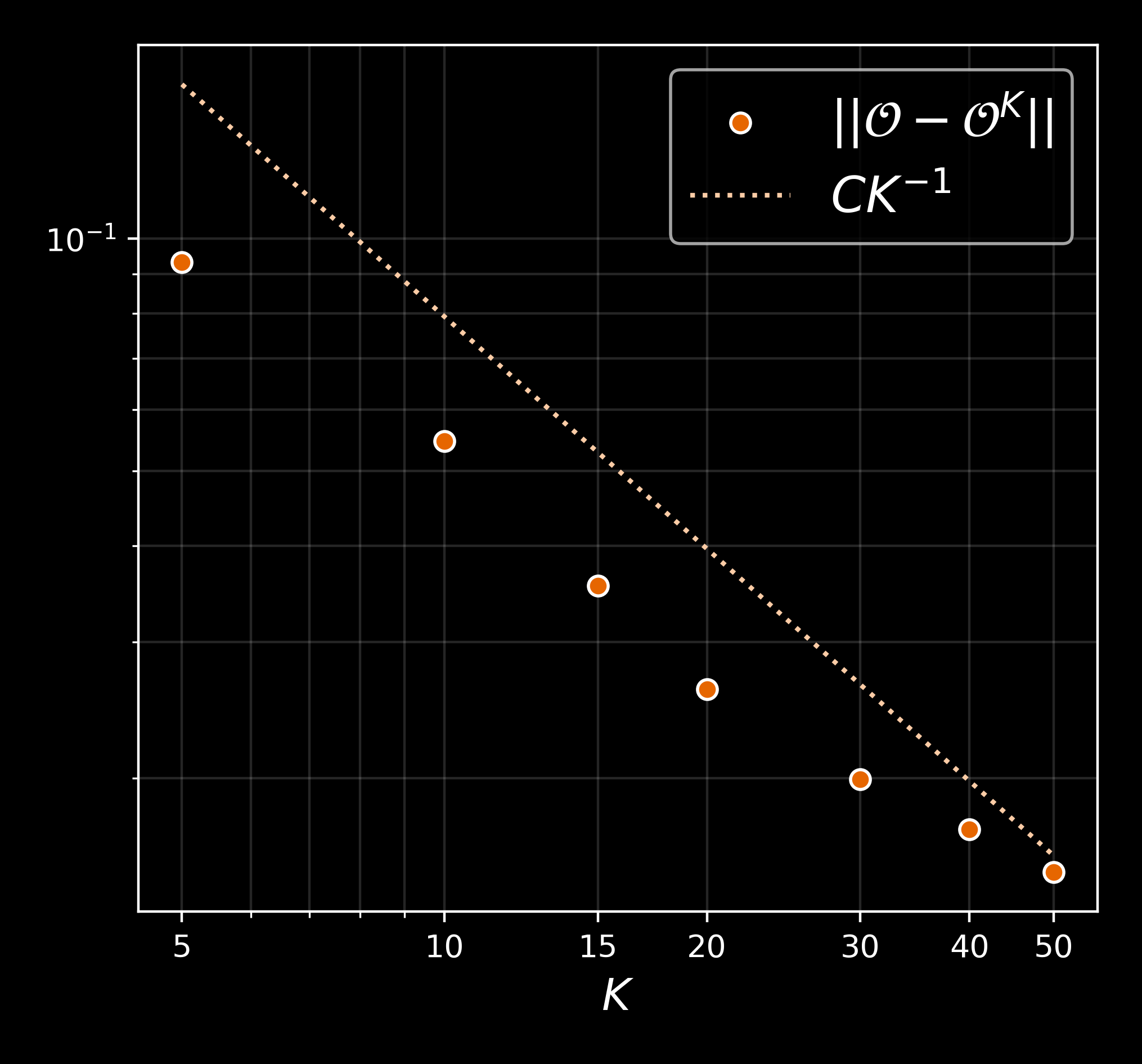}
\end{subfigure}
\end{figure}

\subsection{Local Occupied Volatility Model}\label{sec:LOV}
We conclude this section with a financial example where the diffusion coefficient, termed \emph{volatility}, depends on the exponential time occupation flow. Specifically, 
introduce the  simplified local occupied volatility (LOV) model \cite{TissotLOV}, 
    \begin{align}\label{eq:LOV}
    d\calO_t &= \delta_{X_t} e^{\kappa t}dt, \quad \quad  \calO_t(\R) = \frac{e^{\kappa t} - 1}{\kappa}=:\Lambda_t,\\[1em] 
        \frac{dX_t}{X_t} &= \sigma(\calO_t,X_t) \, dW_t, \\[1em]
      \sigma^2(\calO_t,X_t) &= v_{\text{loc}}(X_t) + \frac{1}{\Lambda_t}\int_{\R^+}\ell(x/X_t)\calO_t(dx), \label{eq:occVariance}
    \end{align}
    for some time-homogeneous \textit{local variance}  $v_{\text{loc}}:\R_+\to \R_+$. The function $\ell:\R_+\to \R_+$ describes the sensitivity of instantaneous variance with respect to path-dependent shocks, captured by the occupation measure $\calO_t$. In this example, we choose
\begin{equation}\label{eq:LV_Sensitivity}
        v_{\text{loc}}(x) = \alpha (x/x_0 - 1)^2 + \beta (x/x_0 - 1) +  \gamma, \quad \ell(y) = \delta\tanh((y-1)/\varepsilon).
    \end{equation}
That is, $v_{\text{loc}}(X_t)$ is a quadratic function of the simple return $X_t/x_0 - 1$ over the period $[0,t]$. The local variance achieves a minimum value of $v_{\text{min}} := \gamma  - \frac{\beta^2}{4\alpha}$. The local volatility $\sigma_{\text{loc}} := \sqrt{v_{\text{loc}}}$ and sensitivity function $\ell$ are displayed in \cref{fig:LV_and_sensitivity} for the parameters $(\alpha,\beta,\gamma,\delta,\varepsilon) = (1,-0.1,0.01, \frac{v_{\text{min}}}{2},0.1)$. Similar to \cite{TissotOP}, setting  $\delta < v_{\text{min}}$ ensures that the occupied variance \eqref{eq:occVariance} remains positive. 

The results are summarized in \cref{fig:simulationsLOV,fig:pathwiseErrorLOV,fig:convergenceLOV} with the aforementioned parameters. A reference solution is obtained by setting $X\approx X^{\overline{K}}$ with $\overline{K} = 100$. The noise observed in the pathwise errors (\cref{fig:pathwiseErrorLOV}) stems from fluctuations in the volatility coefficients across truncation levels. This is in contrast with the related  plot  for Cranston-Le Jan's diffusion (\cref{fig:pathwiseErrorCranstonLeJan}), where the solution only affects the finite variation part of the dynamics. 

\begin{figure}[H]
    \centering
      \caption{Local volatility $\sigma_{\text{loc}} = \sqrt{v_{\text{loc}}}$ (left) and local sensitivity function $\ell$ given in \eqref{eq:LV_Sensitivity} with $(\alpha,\beta,\gamma,\delta,\varepsilon) = (1,-0.1,0.01, \frac{v_{\text{min}}}{2},0.1)$.}
    \label{fig:LV_and_sensitivity}

\begin{subfigure}[b]{0.495\textwidth}
        \centering
    \includegraphics[width=0.92\linewidth,height = 2.1in]{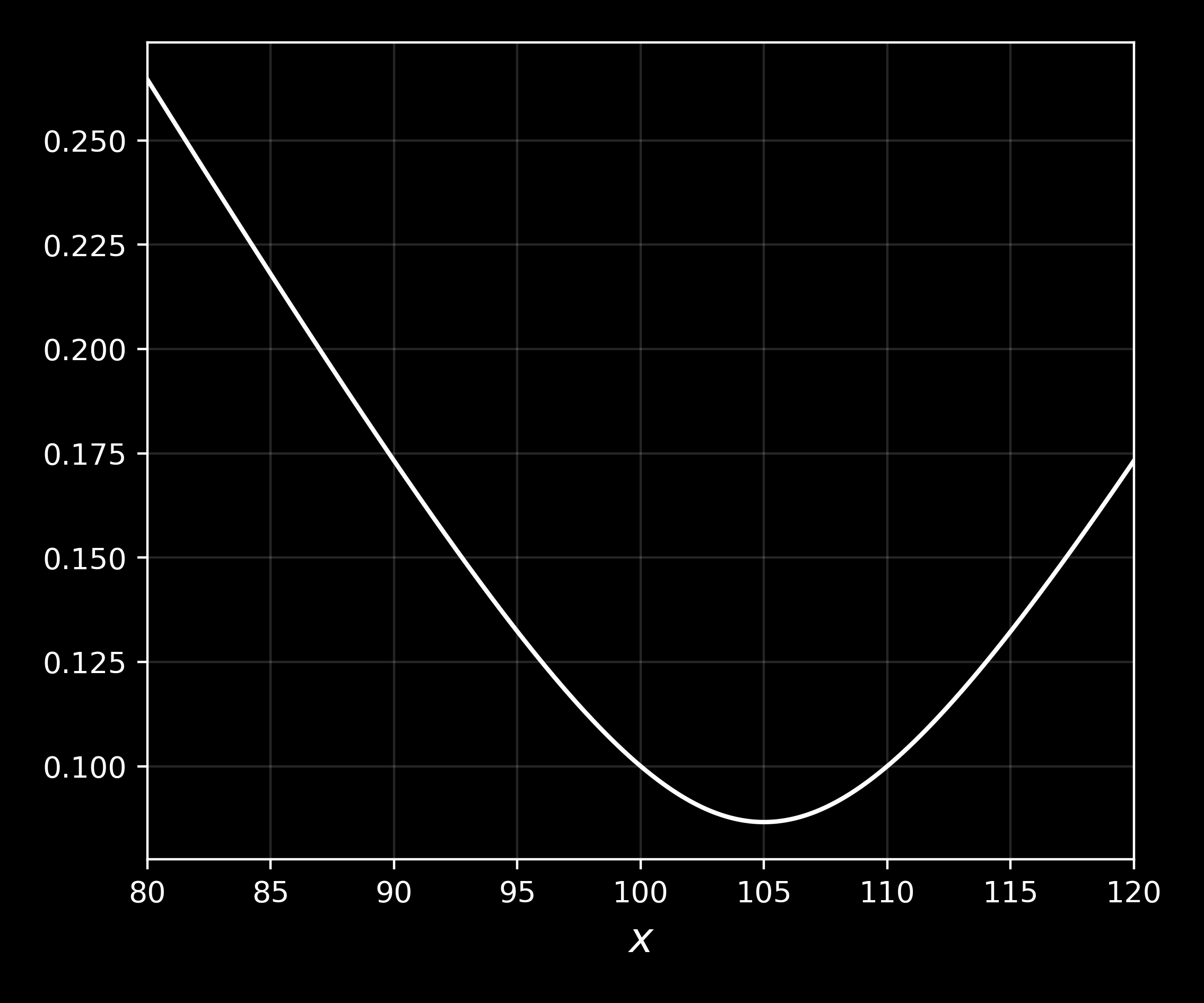}
\end{subfigure}
\begin{subfigure}[b]{0.495\textwidth}
 \centering
    \includegraphics[width=0.92\linewidth,height = 2.1in]{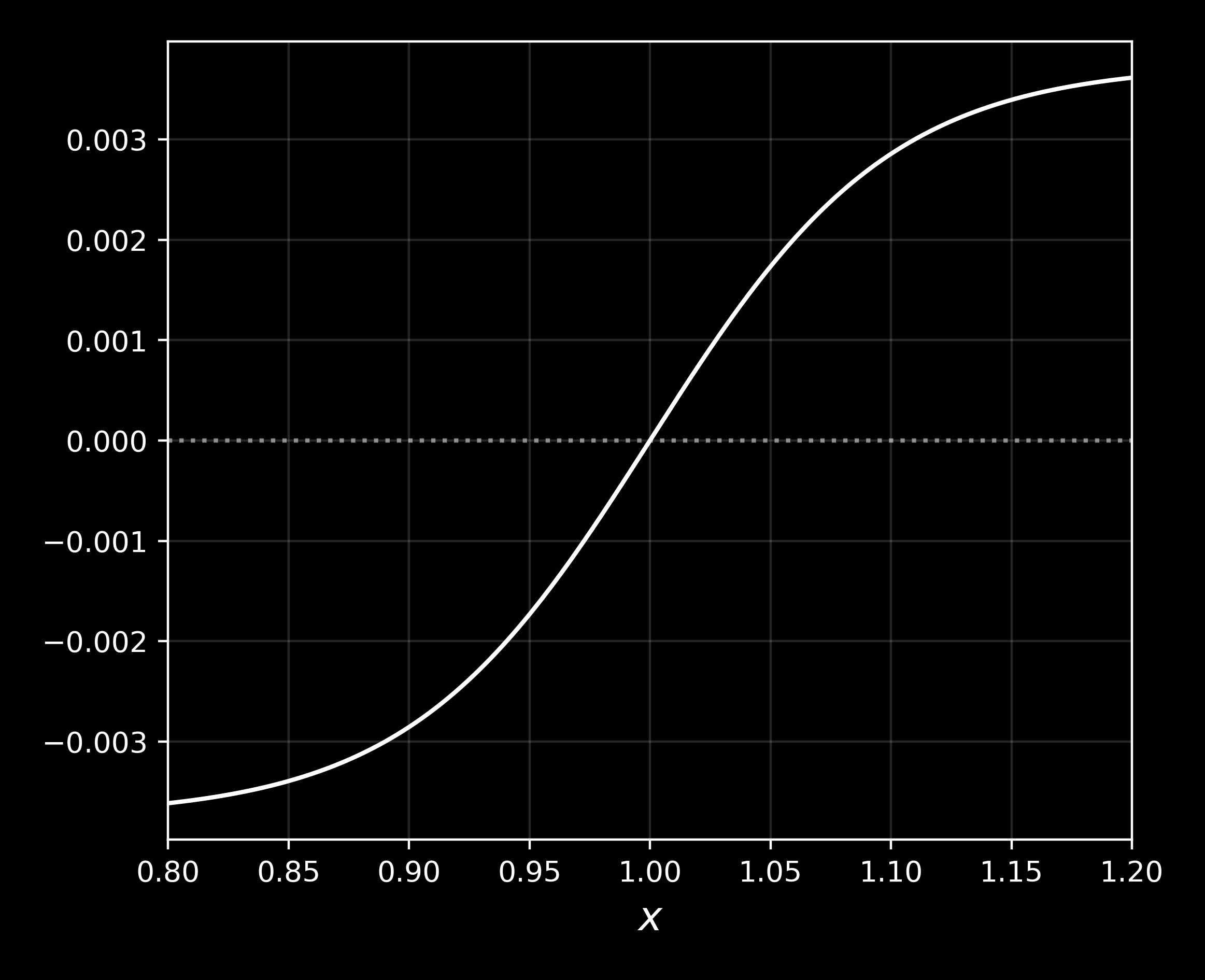}
\end{subfigure}
\end{figure}

\begin{figure}[H]
    \centering
    \caption{Sample path of the LOV model \eqref{eq:LOV} for varying truncation levels.}

    \label{fig:simulationsLOV}

\begin{subfigure}[b]{0.495\textwidth}
        \centering
       \caption{Full horizon $[0,1]$.}
    \includegraphics[width=0.92\linewidth]{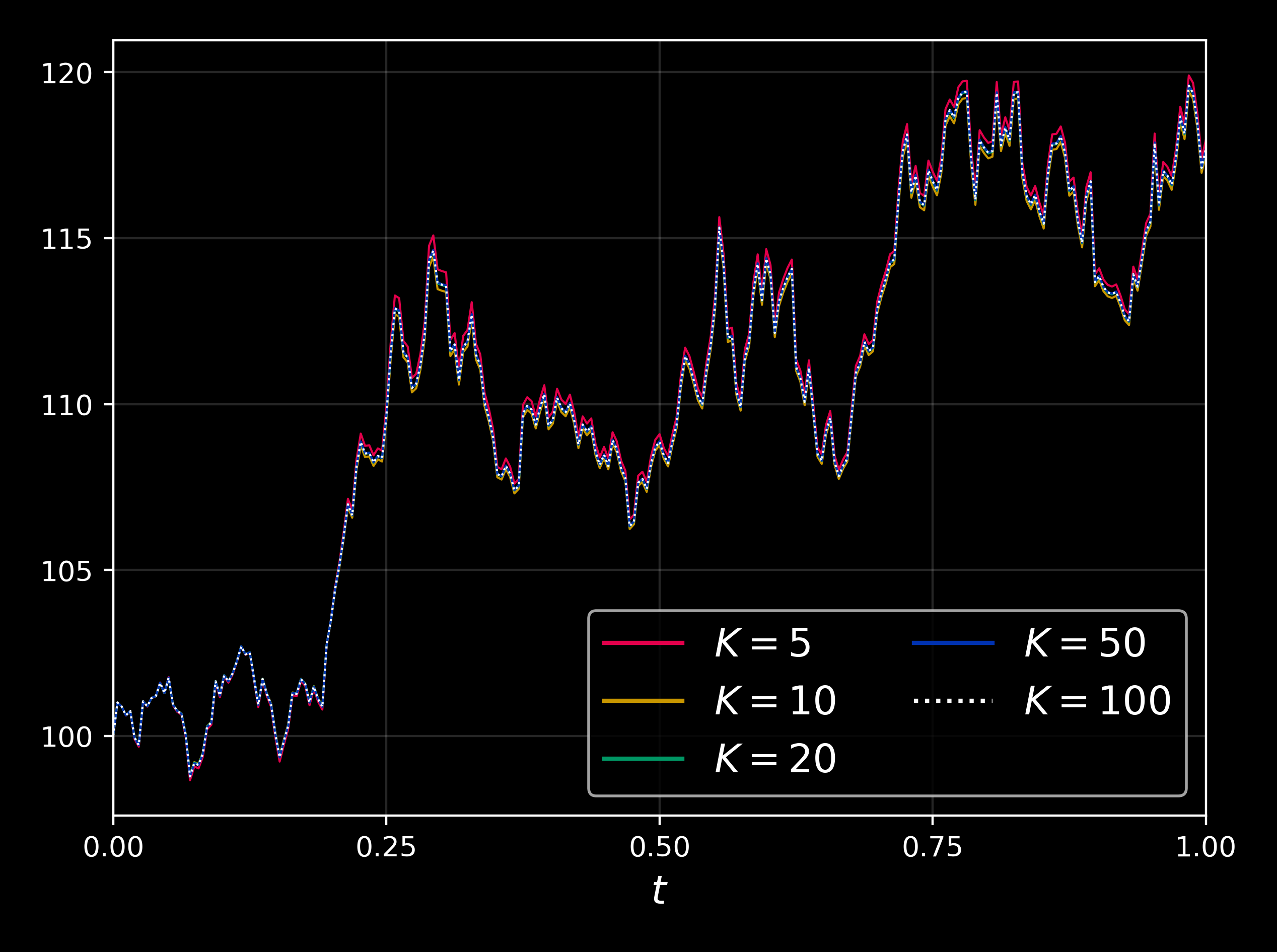}
\end{subfigure}
\begin{subfigure}[b]{0.495\textwidth}
 \centering
\caption{Zoom on $[0.75,1]$.}
    \includegraphics[width=0.9\linewidth]{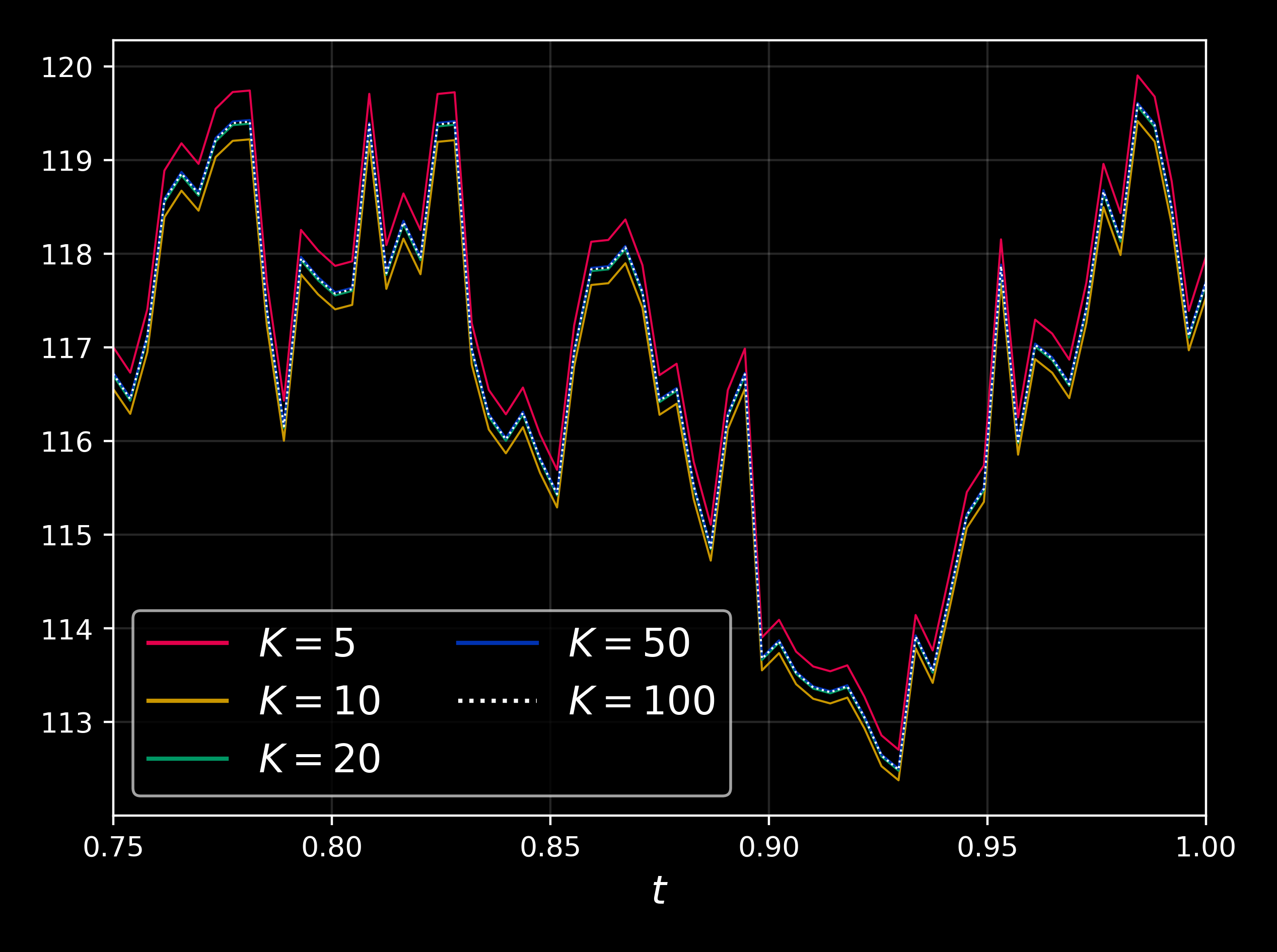}
\end{subfigure}
\end{figure}

\begin{figure}[H]
    \centering
    \caption{Pathwise errors $X^{100}-X^K$ in the LOV model \eqref{eq:LOV} for $J=1000$ simulations and increasing values of $K$.}
   
    \label{fig:pathwiseErrorLOV}
    \includegraphics[width=0.55\linewidth]{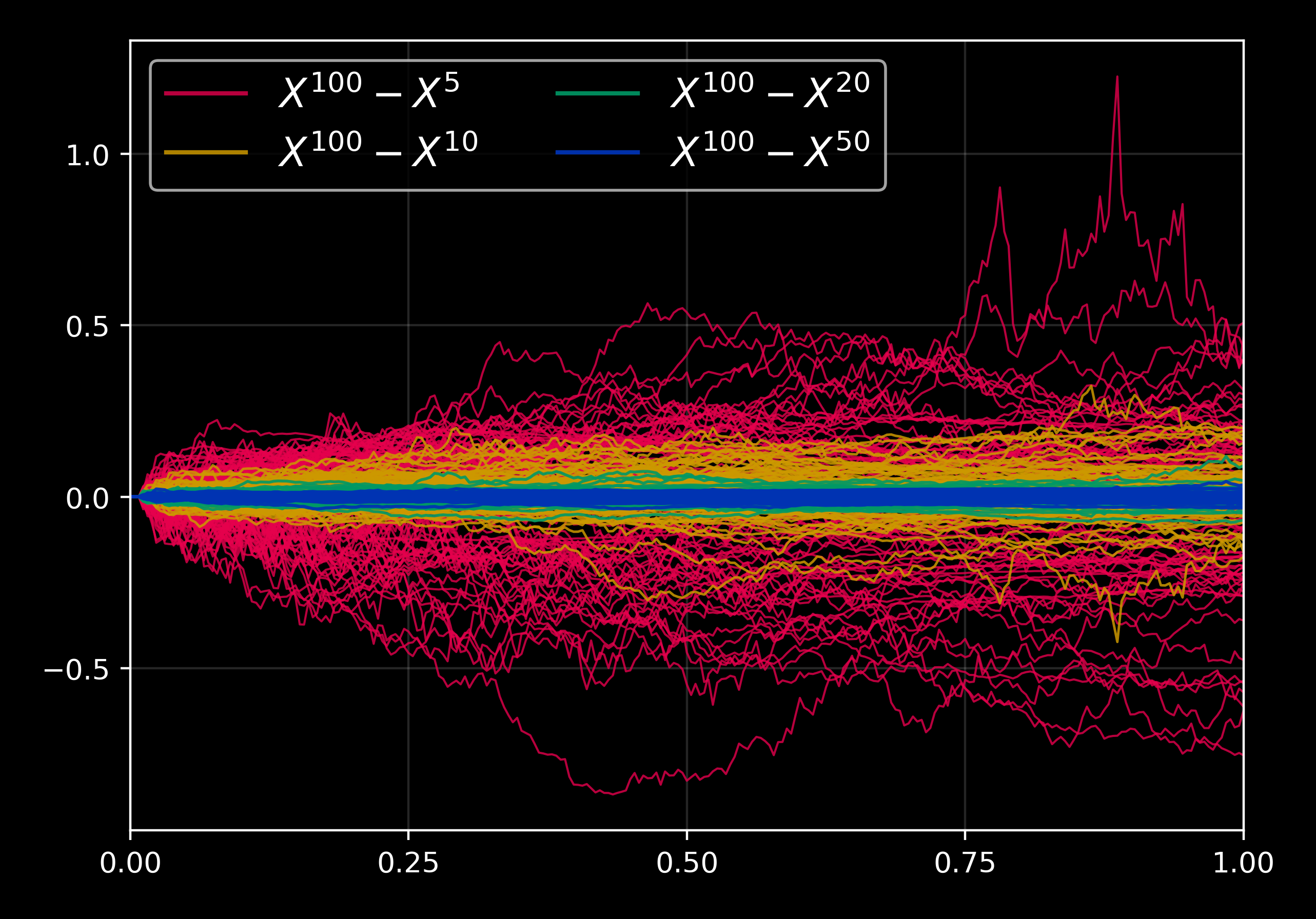}
    \label{fig:sim6}

\end{figure}

\begin{figure}[H]
    \centering
    \caption{Convergence rate for the LOV model \eqref{eq:LOV}-\eqref{eq:occVariance}.}
   
    \label{fig:convergenceLOV}

\begin{subfigure}[b]{0.49\textwidth}
        \centering
    \includegraphics[width=0.82\linewidth]{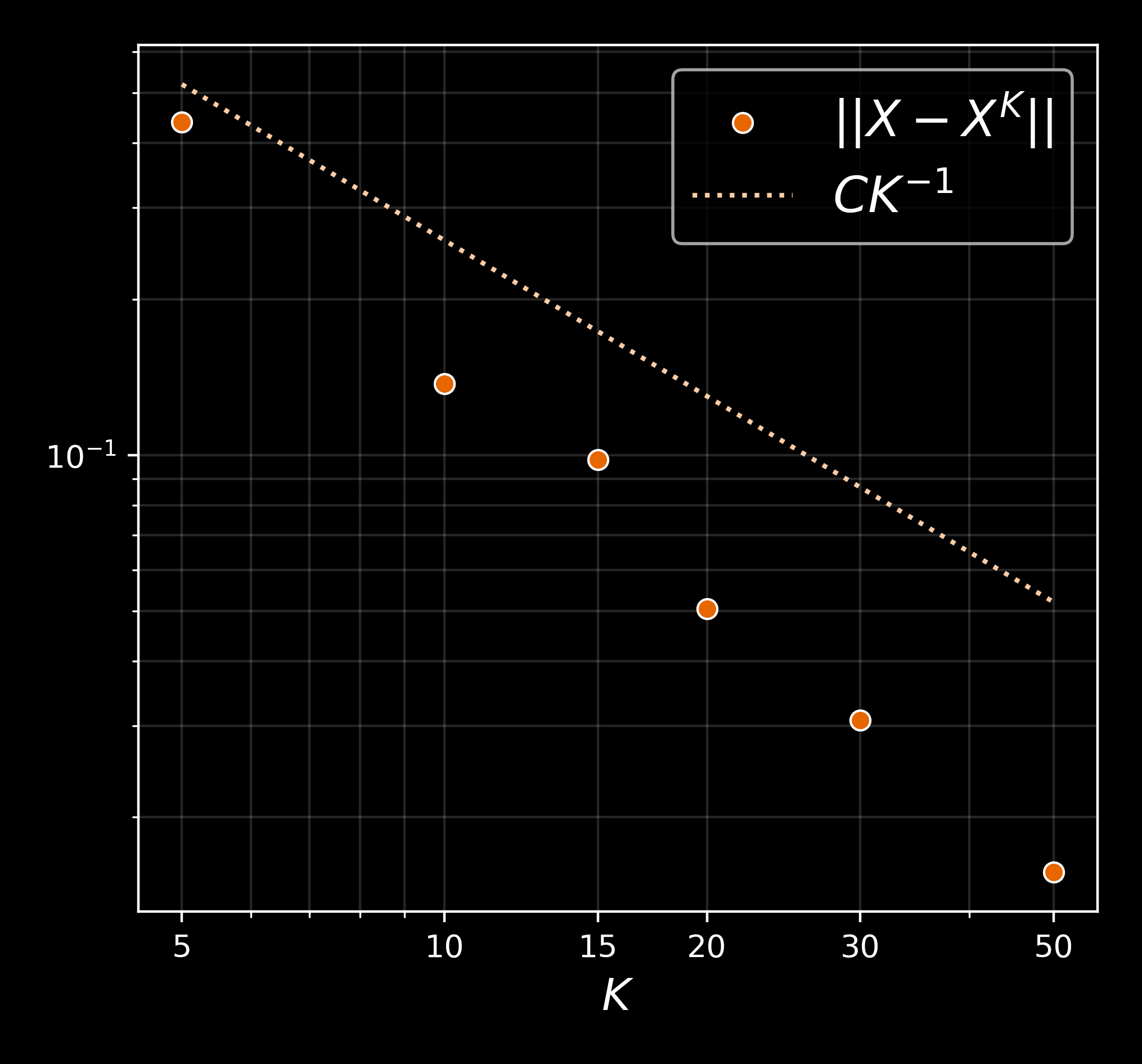}
    \label{fig:sim3}
\end{subfigure}
\begin{subfigure}[b]{0.49\textwidth}
 \centering
    \includegraphics[width=0.8\linewidth]{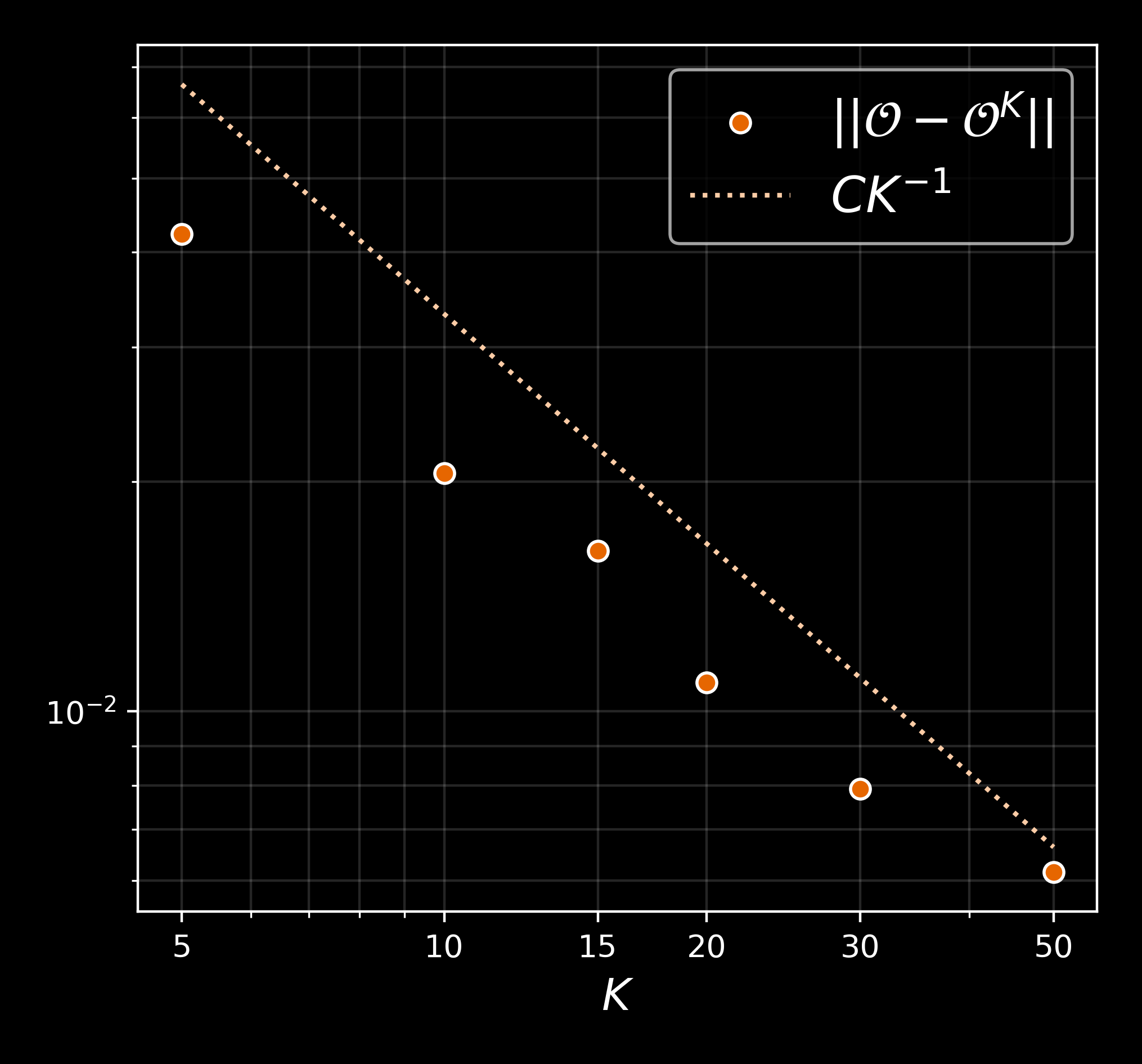}
    \label{fig:sim6}
\end{subfigure}
\end{figure}

\subsection{Alternative Schemes and Comparison} \label{sec:comparison}

We compare the proposed cylindrical projection
with two natural alternatives: the direct-history Euler scheme and a Fourier-moment
scheme. Let \(t_n=n\delta t\), \(n=0,\ldots,N\), and let \(M_K:=N_K+1\) denote the number of
coordinates in the cylindrical state \(Z^K\). For a tensor-product grid with \(K\) cells
per coordinate, \(M_K\asymp K^d\).

\paragraph{Direct-history scheme.} This approach consists of discretizing the occupation flow according to 
\[
O^N_{t_n}=\sum_{j=0}^{n-1}\delta_{X_{t_j}} \delta t,
\]
and evaluates the coefficients using the full past trajectory. For the
regularized Raimond interaction (Section~\ref{sec:raimond}), 
\begin{equation}\label{eq:raimondDrift}
    b(\mo,x)=\beta\int_{\mathbb R^d}\phi_\varepsilon(y-x)\mo(dy),
\qquad
\phi_\varepsilon(u)=\frac{u}{\sqrt{\varepsilon+|u|^2}},
\end{equation}
the direct Euler drift is
\[
b_n^{\rm hist}
=
\beta \sum_{j=0}^{n-1}\phi_\varepsilon(X_{t_j}-X_{t_n})\delta t.
\]
Thus, step \(n\) requires \(O(nd)\) kernel operations and one path requires
$ \sum_{n=1}^{N} n = O(N^2) $
kernel evaluations, or \(O(N^2d)\) arithmetic operations up to constants. The memory
requirement is \(O(Nd)\) if the whole trajectory must be retained. This direct scheme is
a useful benchmark, since it introduces no spatial projection error beyond the time
discretization error of Euler--Maruyama \cite{KloedenPlaten}. However, its quadratic
cost becomes prohibitive for fine time grids or long horizons. In some
special models, such as the Cranston--Le Jan example, algebraic simplifications reduce
the cost through a running barycenter, but such reductions are model-specific and are
not available for a general occupied diffusion. 

The cylindrical projection replaces \(O_t\) by the finite-dimensional vector
\[
Z^K_t=(\calO_t(f^K_0),\ldots,\calO_t(f^K_{N_K})).
\]
On the time grid, Algorithm~\ref{alg:simOSDE} updates $Z^K$ according to 
\[
Z^{K,n+1}
=
Z^{K,n}
+
\lambda^K(Z^{K,n},X^K_{t_n})f^K(X^K_{t_n}) \delta t,
\]
and then updates \(X^K\) using \(b^K\) and \(\sigma^K\). For indicator boxes, the
occupation update is local, typically \(O(1)\) per step; for a smooth partition of unity,
the update cost is proportional to the number of nonzero basis functions at the current
point. In the self-interacting kernel examples, the dense drift evaluation  has the form
\[
b_n^K
=
\beta\sum_{i=0}^{N_K} Z_i^{K,n}\,
\phi_\varepsilon(x_i^K-X^K_{t_n}),
\]
and costs \(O(M_Kd)\) per step, hence \(O(NM_Kd)\) per path. The online memory is
\(O(M_K)\), apart from the optional storage of the output trajectory. Therefore, when
\(M_K\ll N\), the cylindrical scheme offers a substantial reduction in computational cost compared to the direct-history approach. Even when \(M_K\) is comparable to or larger than \(N\), the method still has two
important advantages: the simulated process remains finite-dimensional and Markovian,
and the projection error is deterministic and controlled by Theorem~\ref{thm:convergence}.

\paragraph{Fourier scheme.} An alternative approximation leverages Fourier moments of the occupation flow. 
Specifically, suppose that the
occupation dependence is linear and translation-invariant as in \eqref{eq:raimondDrift}, that is, 
$
b(\mo,x)=\int_{\mathbb R^d}\phi(y-x)\mo(dy).$  
After localization or periodization on a compact domain, the function $\phi$ admits the truncated Fourier representation
\begin{equation}\label{eq:Fourier}
    \phi(u) \approx \phi_L(u) := \sum_{z_\ell \in \mathcal{I}_L}
a_\ell e^{i z_\ell \cdot u},
\end{equation}
with frequencies $z_\ell$ on the truncated lattice $\mathcal{I}_L = \{-L, \dots, L\}^d \setminus \{0\}$. The past trajectory is then captured by the Fourier moments
\[
Z_\ell^n
=
\sum_{j=0}^{n-1}e^{i z_\ell \cdot X_{t_j}} \delta t,
\qquad
Z_\ell^{n+1}
=
Z_\ell^n+ e^{i z_\ell \cdot X_{t_n}}  \delta t.
\]
From the occupation time formula, the drift is estimated  as
$$b(\calO_{t_n},X_{t_n}) \approx 
\sum_{z_\ell \in \mathcal{I}_L}
a_\ell e^{-i z_\ell \cdot X_{t_n}}Z_\ell^n .$$
The resulting cost is $O(N M_L d)$ per path, where $M_L := |\calI_L|= (2L+1)^d - 1,$ and the active memory is $O(M_L)$. For higher dimensions ($d \ge 4$), the deterministic lattice can be replaced by random Fourier Features 
to break the curse of dimensionality. Similar randomization techniques can be used for the cylindrical projections by optimally selecting the centers of the partition of unity, e.g. according to a low-discrepancy sequence. 

The Fourier
coefficients \(a_\ell\) in \eqref{eq:Fourier} may be computed offline by spectral methods or FFT-based
procedures; see \cite{Trefethen2000,Boyd2001,CooleyTukey1965}. This approach can be
very efficient when \(\phi\) is smooth, periodic, and of convolution type. However, the scheme is less effective for singular or sharply localized kernels, which require numerous high-frequency modes to resolve. 
It also relies on global oscillatory basis functions, introduces truncation
or periodic-boundary effects, and is not directly applicable to general nonlinear
dependencies of \((\lambda,b,\sigma)\) on the occupation measure. For instance, the LOV
model in Section 4.3 depends on the occupation measure through a state-dependent
volatility functional, which is naturally handled by cylindrical projection but is not a
simple translation-invariant Fourier convolution.

Fourier moments
may be faster than cylindrical projections when \(M_L\ll M_K\) and the kernel is smooth and
translation-invariant. The advantages of the cylindrical projection are of a different nature: it
provides a model-independent finite-dimensional Markov approximation of the occupation
flow, preserves the measure-based interpretation of  path dependence, applies to
nonlinear and non-convolutional coefficients. Furthermore, it  comes with the strong
projection error estimates established in Sections~\ref{sec:convergence}.

\begin{table}[htbp]
\centering
\caption{Computational complexity and limitations of discretization schemes.}
\label{tab:comparison}
\begin{tabular}{c|c|c}
\hline
\textbf{Scheme} & \textbf{Cost per path} & \textbf{Limitations} \\
\hline
Direct history         & $O(N^2d)$   & quadratic cost and full path in RAM \\
Fourier moments        & $O(NM_Ld)$    & convolution structure only and $M_L \asymp L^d$\\ 
Cylindrical projection & $O(NM_Kd)$  & projection dimension $M_K \asymp K^d$ \\
\hline
\end{tabular}
\end{table} 

Table~\ref{tab:comparison} provides a comparative summary of the discretization schemes, while Table~\ref{tab:runtimeCompError} details their empirical performance for Raimond's self-interacting diffusion with $\beta = 5$ and $d\in \{2,3\}$. Strong errors are evaluated against the direct-history scheme using the norm $\|\cdot\|_{\Q}$ from  \eqref{eq:normX}.  For the Fourier scheme with $d= 2$, we use a spectral frequency limit of  $L = 5$, leading to $M_L = (2L + 1)^2 - 1 =  120$ modes, hence $2M_L = 240$  scalar values. 
We align the cylindrical scheme's memory usage by setting $K = 15$ intervals per dimension for the indicator functions $f_k^K = \mathds{1}_{A_k^K}$, giving $M_K = K^2 + 1 = 226$ scalar values. When $d = 3$, we choose $L = 3$ and $K = 8$, leading to $684$ and $513$ scalar values in active memory, respectively.\footnote{Alternative partitions of unity, such as linear and cubic splines, were also tested for the cylindrical scheme; these provided similar accuracy to indicator functions but resulted in higher runtimes.} 
 Table~\ref{tab:runtimeCompError} shows that the cylindrical projection scheme outperforms the Fourier method in both accuracy and execution speed. 

Runtimes are plotted in the left panel of Figure~\ref{fig:performance_metrics}, highlighting the quadratic growth of the direct-history scheme.  Interestingly, the Fourier scheme is more than twice as slow as the cylindrical projection despite their aligned memory constraints. This behavior is driven by the repeated evaluation of complex exponentials, which require multiple CPU clock cycles compared to the basic floating-point operations utilized in the cylindrical projection. 

\begin{table}[htbp]
\centering
\caption{Runtime per path (in milliseconds) and strong errors against the direct-history scheme for Raimond's self-interacting diffusion ($\beta = 5$) in dimensions $d \in \{2, 3\}$.}
\label{tab:runtimeCompError}
\begin{tabular}{c|c|c|c|c|c}
\hline
\multicolumn{6}{c}{\textbf{d = 2}} \\
\hline
 & \multicolumn{3}{c|}{\textbf{Runtime} (milliseconds)} & \multicolumn{2}{c}{\textbf{Strong errors}} \\
\hline
\textbf{$N$} & {Direct-history} & {Fourier} & {Cylindrical} & {Fourier} & {Cylindrical} \\
\hline
128  & 0.06 & 0.56  & 0.26  & 0.30 & 0.09  \\     
256  & 0.23  & 1.20 & 0.51   & 0.30  & 0.10 \\      
512  & 1.05  & 2.38  & 1.07  & 0.31 & 0.10  \\     
1024 & 5.02  & 4.87 & 2.12   & 0.31 & 0.09  \\     
2048 & 19.94 & 9.80 & 4.56   & 0.29 & 0.10  \\    
4096 & 82.50 & 19.54 & 9.08  & 0.31 & 0.10  \\ 
\hline
\multicolumn{6}{c}{\textbf{d = 3}} \\
\hline
 & \multicolumn{3}{c|}{\textbf{Runtime} (milliseconds)} & \multicolumn{2}{c}{\textbf{Strong errors}} \\
\hline
\textbf{$N$} & {Direct-history} & {Fourier} & {Cylindrical} & {Fourier} & {Cylindrical} \\
\hline
128    & 0.07         & 1.94     & 0.88              & 0.44       & 0.14         \\     
256    & 0.36            & 3.88    & 1.75             & 0.44     & 0.14         \\      
512    & 1.47              & 7.68  & 3.84               & 0.44   & 0.13         \\      
1024   & 6.17           & 16.29  & 6.99                  & 0.43  & 0.14          \\
2048   & 27.44           & 31.31  & 14.93                 & 0.43  & 0.15          \\   
4096   & 118.32        & 66.42    & 29.48                 & 0.42   & 0.14          \\
\hline
\end{tabular}
\end{table}

\begin{remark}
    More generally, the structural mechanism behind the Fourier scheme can be extended to general kernels $\phi = \phi(y,x)$  by leveraging approximations via dense families of separable functions. If the bivariate kernel admits a truncated expansion in terms of products or sums of univariate functions---such as those arising from low-rank tensor-train decompositions or  polynomial expansions techniques---it can be written as 
    $$\phi(y,x) \approx \sum_{\ell=1}^L \phi_{\ell,1}(y) \phi_{\ell,2}(x),$$ 
    for some truncation level $L\in \N$. Under this separation of variables, the past trajectory can again be compressed online into running moments $Z_\ell^n = \sum_{j=0}^{n-1} \phi_{\ell,1}(X_{t_j})\delta t$, which satisfy the  $O(1)$  update $Z_\ell^{n+1} = Z_\ell^n + \phi_{\ell,1}(X_{t_n})\delta t$. The resulting drift $
b_n^L =  \sum_{\ell=1}^L Z_\ell^n \phi_{\ell,2}(X_{t_n})$ 
circumvents the $O(N^2)$ direct-history bottleneck to achieve an overall  complexity of $O(N L d)$ per path, with an active memory  of $O(L)$. While highly efficient for smooth kernels where a low-rank structure can be computed offline, this paradigm is  restricted to linear occupation-dependencies and scales poorly if the non-convolutional kernel requires a prohibitive  number of mixing modes to  uniformly control the approximation error. 
\end{remark}

\begin{figure}[htbp]
    \centering

    \caption{Computational metrics for Raimond's self-interacting diffusion ($\beta = 5$) across dimensions $d = 2$ (top row) and $d = 3$ (bottom row).}
    \begin{subfigure}[b]{0.49\linewidth}
        \centering
            \caption{Runtime per path (ms).}
        \includegraphics[width=0.8\linewidth]{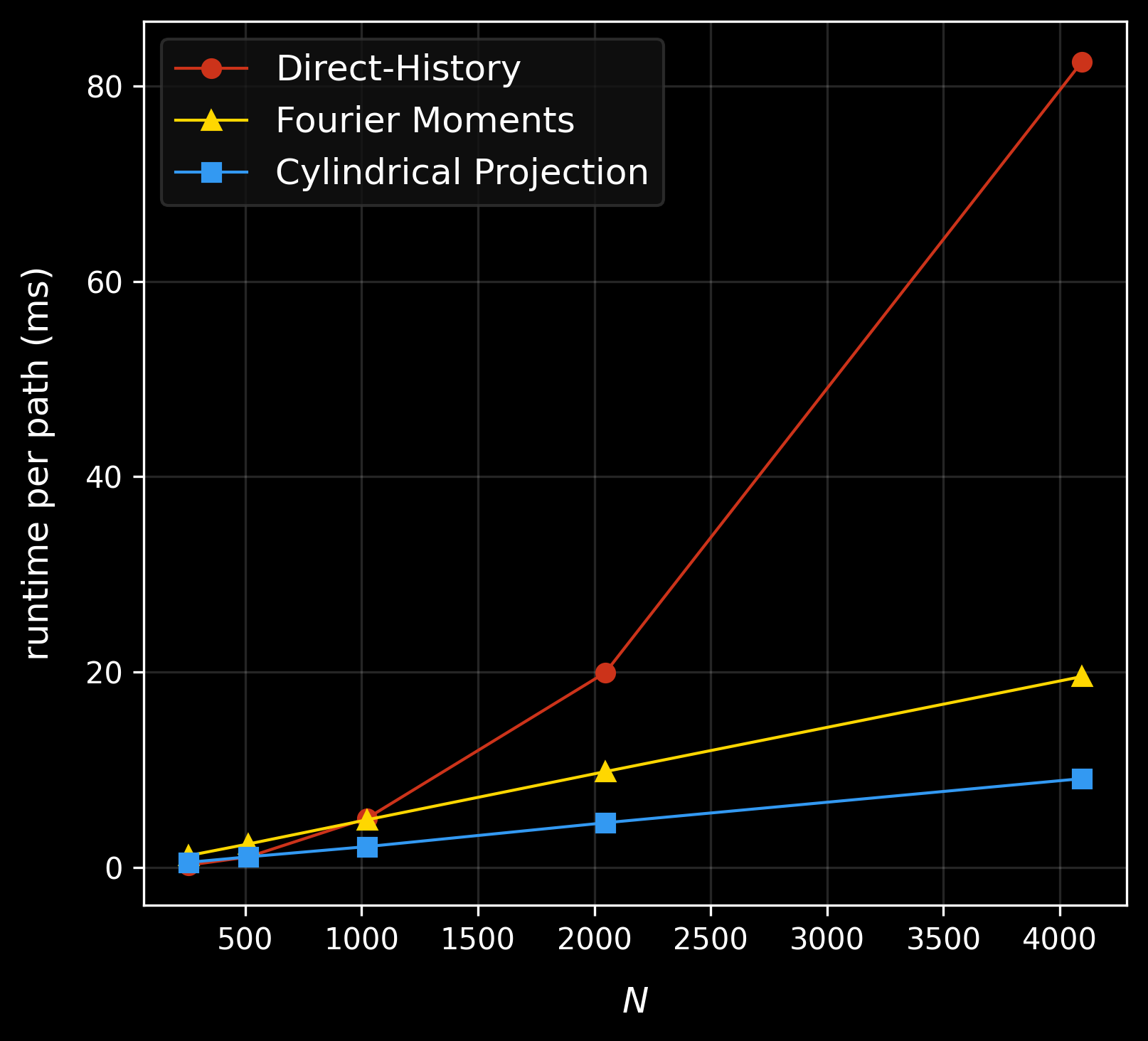}
        \label{fig:runtimeComparison_d2}
    \end{subfigure}
    \hfill
    \begin{subfigure}[b]{0.49\linewidth}
        \centering
          \caption{Active memory per path.}
        \includegraphics[width=0.8\linewidth]{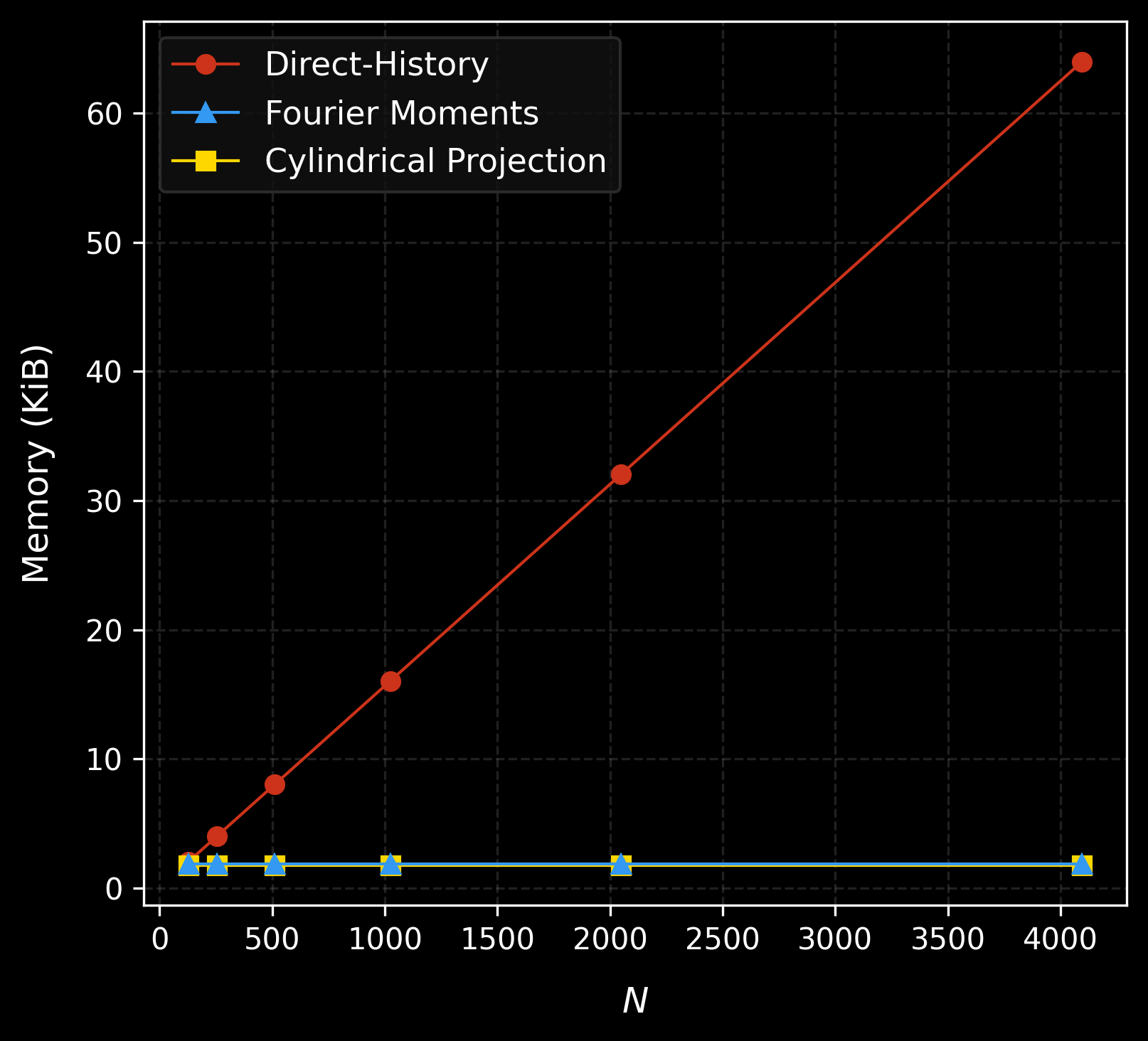}
        \label{fig:memorycomparison_d2}
    \end{subfigure}

    \vspace{1.5em} 
    
    \begin{subfigure}[b]{0.49\linewidth}
        \centering
        \includegraphics[width=0.8\linewidth]{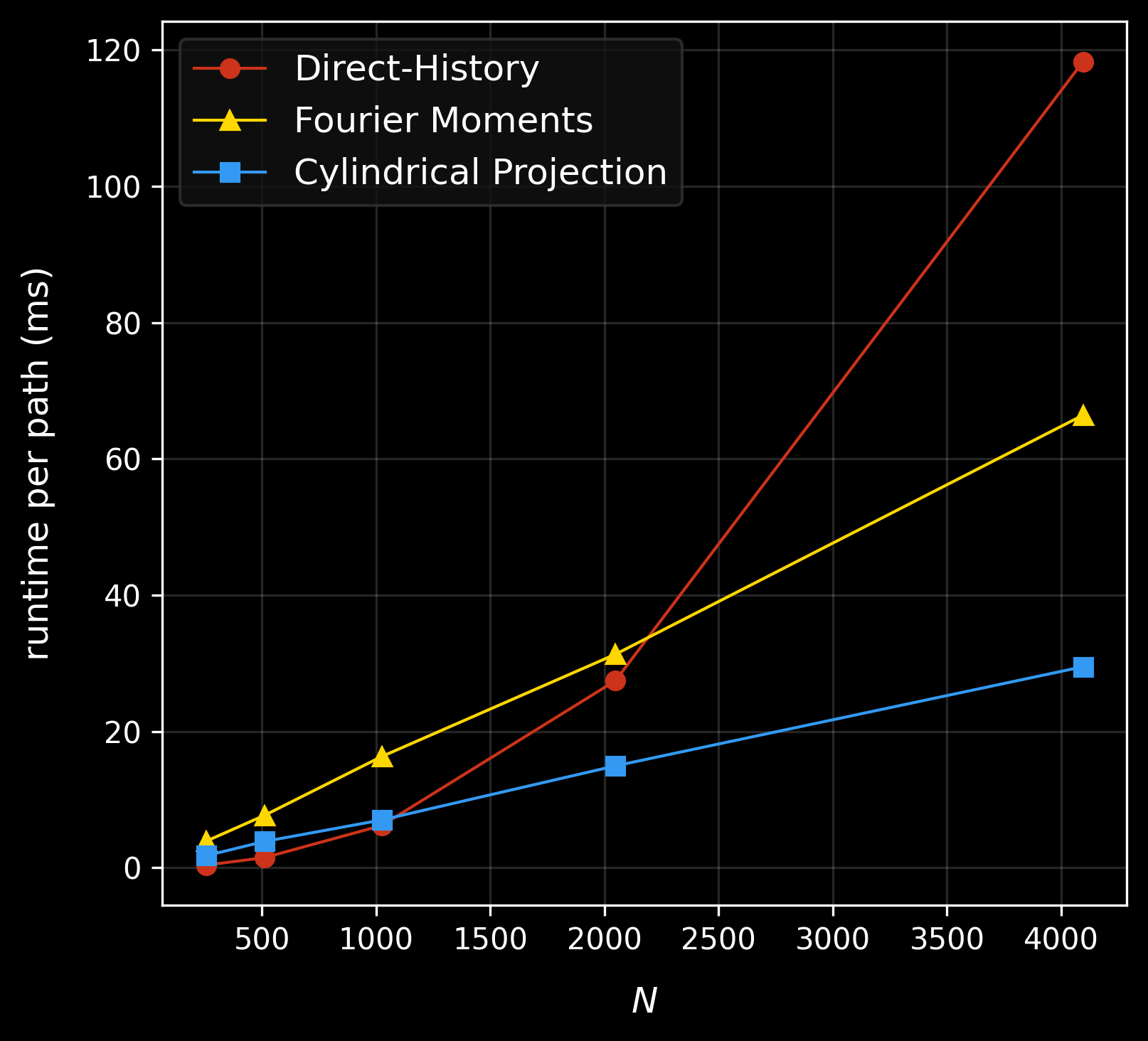}
        
        \label{fig:runtimeComparison_d3}
    \end{subfigure}
    \hfill
    \begin{subfigure}[b]{0.49\linewidth}
        \centering
        \includegraphics[width=0.8\linewidth]{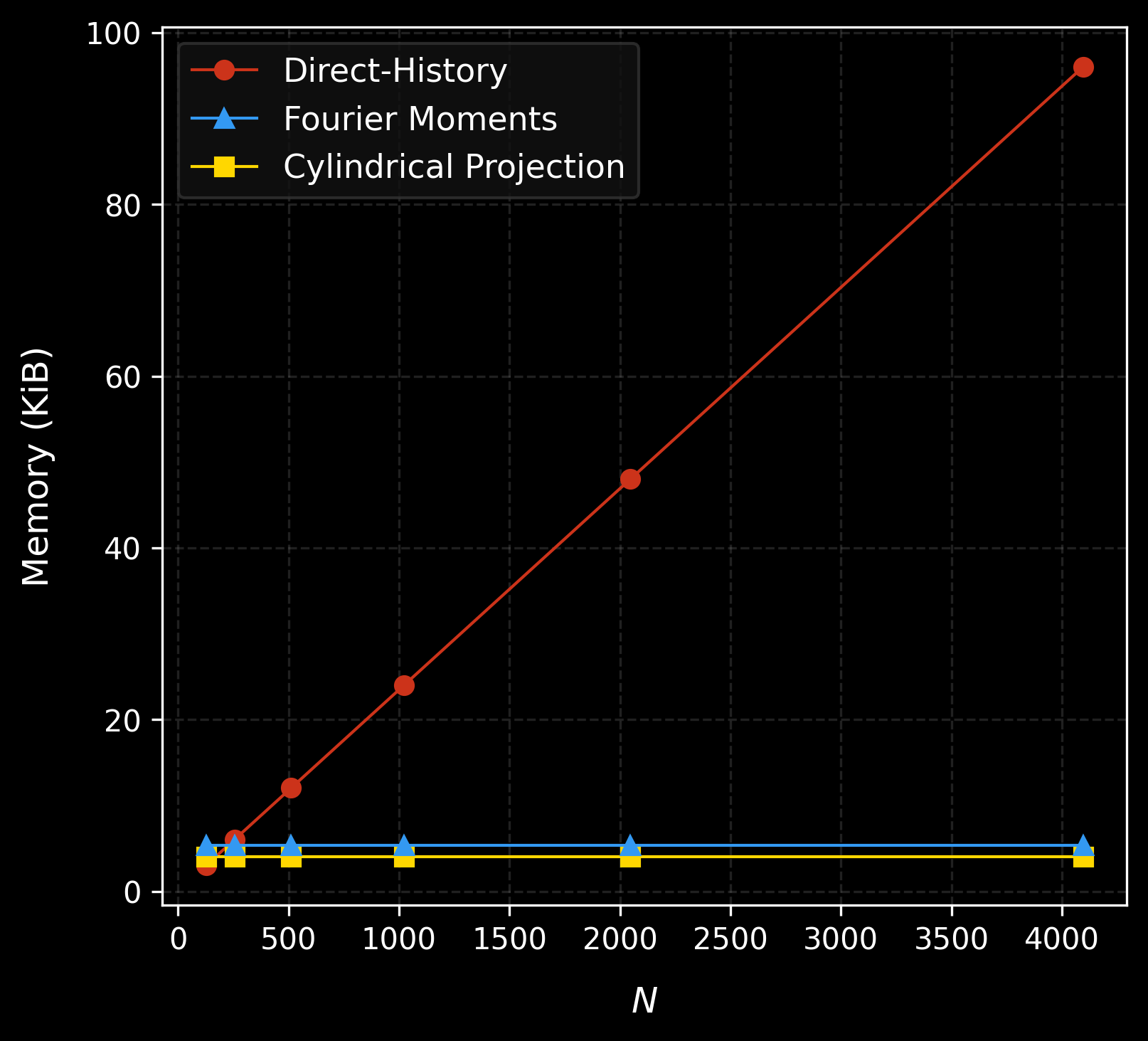}
        \label{fig:memorycomparison_d3}
    \end{subfigure}
    \label{fig:performance_metrics}
\end{figure}

\section{Weak Error Analysis and Monte Carlo Methods} \label{sec:WeakError}



Our main results in \cref{sec:convergence} provide strong error estimates arising from cylindrical projections. These theorems naturally lead to weak convergence rates as outlined next. 

\begin{corollary}\label{cor:weakconvergence}
    Suppose $\varphi:\mathcal{M} \times \mathbb R^d \to \mathbb R$ is Lipschitz with $\mathcal{M}\times \mathbb R^d$ equipped with the product of cylindrical  and Euclidean norm. Then we obtain the weak convergence rate, 
    \begin{align*}
        \left| \E^{\Q}[ \varphi(\calO_T^{\tau_R},X_T^{\tau_R})]-  \E^{\Q}[ \varphi(  \mo^K(Z_T^{K,\tau^K_R}) ,X_T^{K,\tau^K_R})] \right| \leq \frac{C\textnormal{Lip}(\varphi)}{K},
    \end{align*}
    where $\textnormal{Lip}(\varphi)$ denotes the Lipschitz coefficient of $\varphi$, and $C$ is the constant of Theorem~\ref{thm:convergence}. Moreover, similar to Remark~\ref{rmk:truncation}, $(\mo^K(Z_T^{K,\tau^K_R}) ,X_T^{K,\tau^K_R})$  weakly converges to $(\calO_T,X_T)$ at the rate  $1/K$ when $\lambda$ is bounded, and at the rate  $1/(\log K)^{0.35}$ when $\lambda$ is Lipschitz. 
\end{corollary}
\begin{proof}
Write $\varphi_T^R = \varphi(\calO_T^{\tau_R},X_T^{\tau_R})$, $\varphi_T^{R,K} = \varphi(  \mo^K(Z_T^{K,\tau^K_R}) ,X_T^{K,\tau^K_R})$ for brevity. 
Thanks to  the Lipschitz property of $\varphi$, 
\begin{align*}
    \left| \E^{\Q}\big[\varphi_T^R\big]-  \E^{\Q}\big[ \varphi_T^{R,K}\big] \right| &\le  \E^{\Q} \left[ |\varphi_T^R- \varphi_T^{R,K}| \right] \\
    &\le \textnormal{Lip}(\varphi) \E^{\Q} \left[ 
    \lVert (\calO^{\tau_R}_T,X^{\tau_R}_T)-( \mo^K(Z^{K,\tau^K_R}_T) ,X^{K,\tau^K_R}_T) \rVert \right]\\ 
    &\le  \textnormal{Lip}(\varphi)\lVert (\calO^{\tau_R},X^{\tau_R})-( \mo^K(Z^{K,\tau^K_R}) ,X^{K,\tau^K_R}) \rVert_{\Q}.
\end{align*}
The result then follows from Theorem~\ref{thm:convergence}. 
\end{proof}

\paragraph{Applications to Monte Carlo Methods.}
Given a function $\varphi:\calM \times \R^d\to \R$ and a finite horizon $T>0$, suppose we are interested in   estimating the value $$p = \E^{\Q}[\varphi(\calO_T,X_T)],$$ 
where $X$ evolves according to \eqref{eq:OSDE}. Dropping the dependence on the exit time $\tau_R$ for brevity, then combining Monte Carlo simulations with cylindrical projection leads to the estimator
$$p^{K,J} = \frac{1}{J} \sum_{j=1}^J \varphi(\mo^{K}(Z^{K,j}_T),X_T^{K,j}).$$
Introduce  $p^K := \E^{\Q}[\varphi(\mo^K(Z^{K}_T),X_T^K)]$. Classically, $|p^K - p^{K,J}| \sim \frac{1}{\sqrt{J}}$, so together with \cref{cor:weakconvergence}, there exist finite constants $C_1,C_2$ such that  
$$|p - p^{K,J}| \le |p - p^{K}| + |p^K - p^{K,J}| \le \frac{C_1}{K} + \frac{C_2}{\sqrt{J}}.$$
In fact, a third error arises from time discretization which we shall omit here. 

\begin{example} Consider a floating-strike Asian call option in the local occupied volatility (LOV) model from \cref{sec:LOV} with $\kappa = 0$, $x_0 = 100$. That is $\varphi(\calO_T,X_T) = (X_T  - \overline{\calO_T})^+$, where $\overline{\calO_T} = \frac{1}{T}\int_{\R}x\calO_T(dx)$ is the barycenter of $\calO_T$. Since $(\mo,x) \mapsto x - \frac{1}{T}\int_{\R} y\mo(dy)$ is linear in $\mo$, $x$, and $x\to x^{+}$ is $1-$Lipschitz continuous, then $\varphi$ is Lipschitz as well. 
The Monte Carlo prices $p^{K,J}$ and weak errors $|p - p^{K,J}|$ are displayed in \cref{eq:MCAsian} as function of $K$, with $J = 2^{14}$ many simulations.  The exact price $p$ is estimated using $K = 100$ occupation bins. The right panel of \cref{eq:MCAsian} confirms the linear rate of convergence with respect to $K$ established in \cref{cor:weakconvergence}.   
\end{example}

\begin{figure}[H]
    \centering
    \caption{Floating Asian call option  in the LOV model \eqref{eq:LOV}-\eqref{eq:occVariance}.}
   
    \label{fig:MCconvergenceLOV}

\begin{subfigure}[b]{0.49\textwidth}
        \centering
      \caption{Monte Carlo price $K \mapsto p^{K,J}$, $J = 2^{14}$. }
    \includegraphics[width=0.92\linewidth]{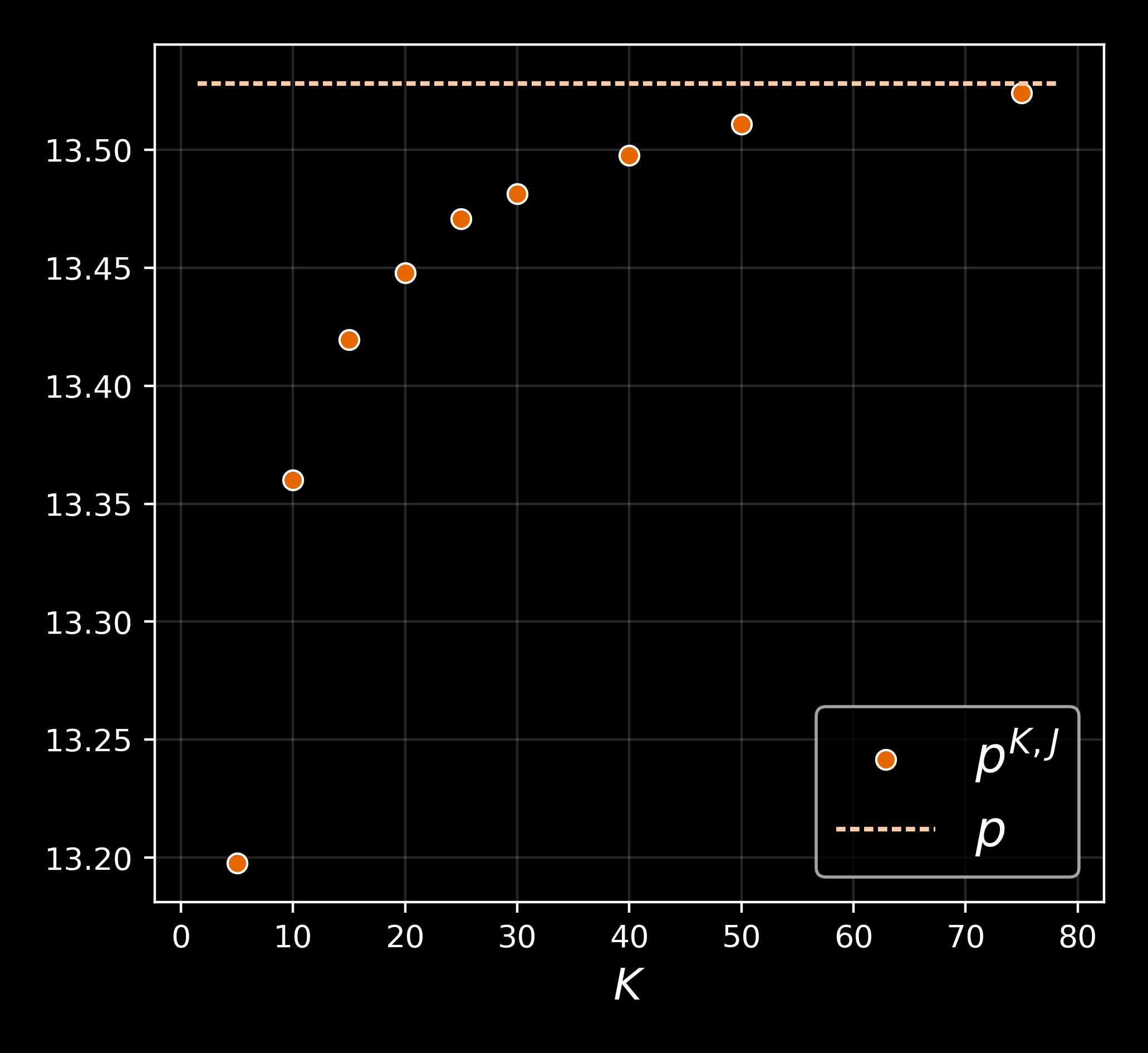}
    \label{fig:sim3}
\end{subfigure}
\begin{subfigure}[b]{0.49\textwidth}
 \centering
\caption{Weak error $K \mapsto |p - p^{K,J}|$}
    \includegraphics[width=0.9\linewidth]{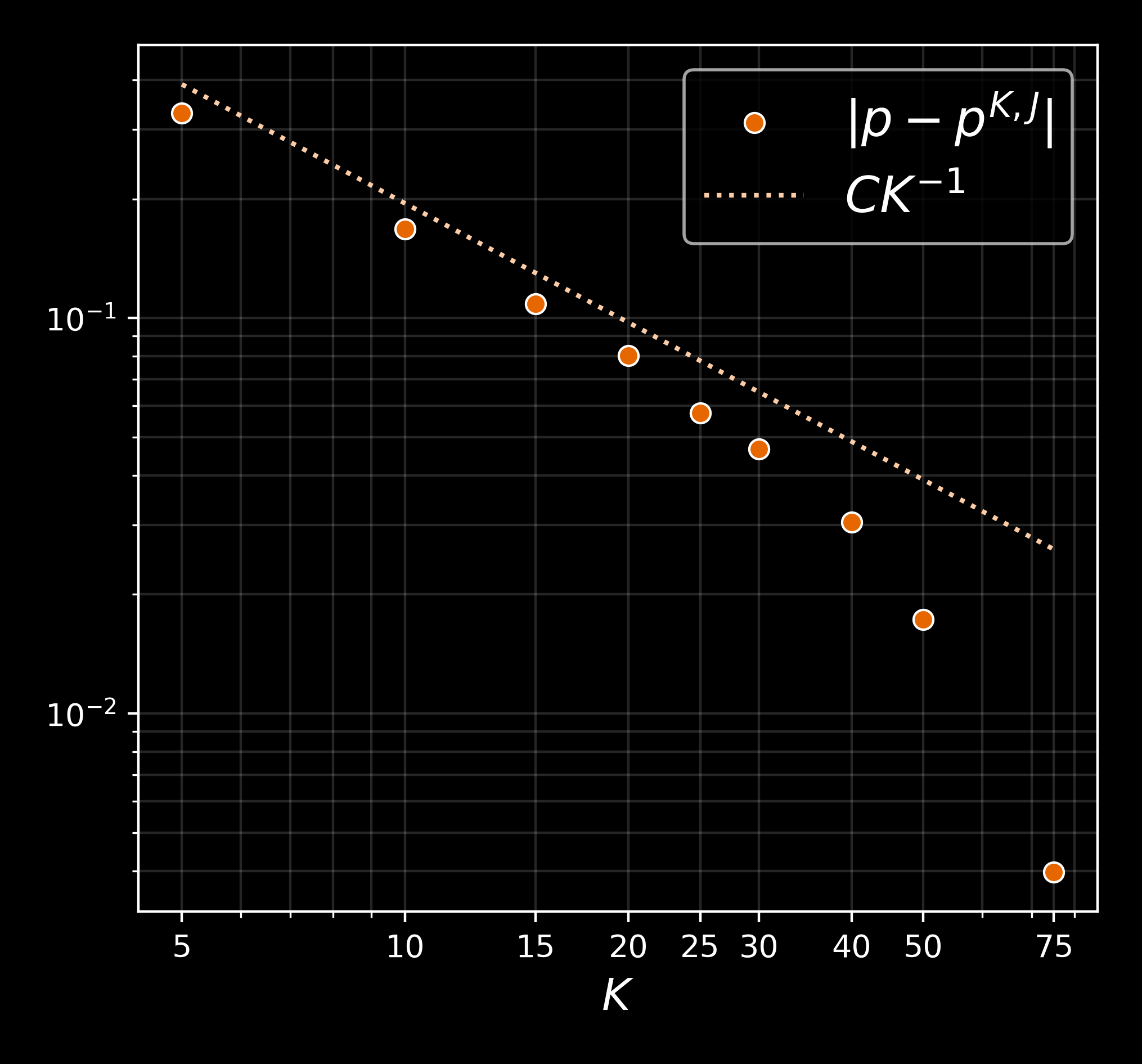}
    \label{fig:sim6}
\end{subfigure}
\label{eq:MCAsian}
\end{figure}

\appendix 

\section{Postponed Proofs} \label{app}
\subsection{Lemma~\ref{lem:stopP}} \label{app:Lemma}
\begin{proof}
Let $Y_t:=\sqrt{1+\lVert(\calO_t,X_t) \rVert^2 }$, and $ \tilde \tau_R:=\inf \{t \geq 0: \, Y_t \geq R \}$, where $(\calO,X)$ is the  strong solution to \eqref{eq:OSDE}. We  prove \eqref{eq:lenglart} through 
\begin{align*}
       \Q( \tau_R \leq T) \leq \Q (\tilde \tau_{R+1} \leq T ) \leq \frac{ e^{\iota \gamma T} (2-\gamma) \E^{\Q}[|Y_0|^{3\gamma}]}{(R+1)^{3 \gamma}(1-\gamma) } \leq \frac{ 4 e^{\iota \gamma T} (2-\gamma) \E^{\Q}[|Y_0|^{3\gamma}]}{R^{3 \gamma}(1-\gamma) },
    \end{align*}
where first and third inequalities are straightforward. So it suffices to show the second inequality. Let us consider $Z_t=e^{-\iota t} Y_t^3$ with some positive $\iota$ to be specified later. According to It\^{o}'s formula, 
 \begin{align*}
     dY_t=&\frac{1}{2Y_t}\left(2X_t \, dX_t +2 \sum_{j=0}^{\infty} \calO_t(g_j) g_j(X_t)\lambda (\calO_t, X_t) \, dt +\sigma(\calO_t,X_t)^2 \, dt \right) \\
     &- \frac{1}{2Y_t^3}\left(X_t^2 \sigma^2(\calO_t,X_t) \, dt \right),
 \end{align*}
 and also that 
 \begin{align*}
     d Z_t= & -\iota Z_t \, dt + 3 e^{-\iota t}  Y_t^2\, dY_t +3 e^{-\iota t} Y_t d \langle Y\rangle_t \\
      = &e^{-\iota t}  \left( -\iota Y_t^3 +   3Y_t^2 \left(\frac{2X_t b(\calO_t,X_t)+2 \sum_{j=0}^{\infty} \calO_t(g_j) g_j(X_t)\iota (\calO_t, X_t)+\sigma(\calO_t,X_t)^2}{2Y_t} \right) \right)dt \\
      & +e^{-\iota t}\left(- \frac{X_t^2 \sigma^2(\calO_t,X_t)} {2Y_t^3}     +\frac{3X_t^2\sigma(\calO_t,X_t)^2}{Y_t} \right) \,dt + \text{local martingale}. 
 \end{align*}

 Note that $\sum_{j=0}^{\infty} \calO_t(g_j) g_j(X_t) \leq \lVert \calO_t \rVert$ and $|X_t|+\lVert \calO_t \rVert+1 \leq 3Y_t$. Thanks to the linear growth assumption~\ref{asm:growthLip}, with large enough $\iota$ which only depends on $b,\sigma, \lambda$,  $Z$ becomes a local supermartingale. As $Z$ is nonnegative, it is a supermartingale. Therefore for any $A \in \mathcal{F}_0$ and stopping time $\tau$, we have that 
 \begin{align*}
     \E^{\Q}[ Z_{\tau} \mathbbm{1}_A] \leq \E^{\Q}[ Z_0 \mathbbm{1}_A]. 
 \end{align*}
 Applying \cite[Lemma 3.2]{GyKr03} with $f_t=Z_t$, $g_t=Z_0$, we get that for any $\gamma \in (0,1)$ and any stopping time $ \tau $
\begin{align*}
    \E^{\Q} \left[ \left(\sup_{t \leq \tau \wedge T} Z_t\right)^{\gamma}\right] \leq \frac{2-\gamma}{1-\gamma} \E^{\Q}[Z_0^{\gamma} ],
\end{align*}
and thus 
\begin{align*}
    e^{-\iota \gamma T} \E^{\Q} \left[ \sup_{t \leq  T} Y_t^{3 \gamma} \right]\leq \frac{2-\gamma}{1-\gamma} \E^{\Q}[Y_0^{3\gamma} ].  
\end{align*}
We then conclude that 
\begin{align*}
    \Q (\tilde\tau_R \leq T ) &=\Q \left( \sup_{t \leq T} |Y_t|  \geq R^{3 \gamma} \right)  \leq \frac{ \E^{\Q} \left[ \sup_{t \leq  T} Y_t^{3\gamma} \right]}{ R^{3\gamma}} \leq \frac{ e^{\iota \gamma T} (2-\gamma) \E^{\Q}[Y_0^{3\gamma}]}{(1-\gamma)R^{3 \gamma} }.
\end{align*}
\end{proof}

\subsection{Proposition~\ref{lem:stopP}} \label{app:Volterra}
\begin{proof}
    Let 
$Z_t := \int_0^t (X_t - X_u)\,du$. Noting that  $dZ_t = t\,dX_t$ and $dX_t = - \beta Z_t\,dt +  dW_t $ leads to the linear SDE
\[
dZ_t  = -\beta t Z_t\,dt + t\,dW_t, \qquad Z_0 = 0 .
\]
Using the integrating factor $e^{\frac{\beta}{2}t^2}$, it is then easily seen that 
$
Z_t
= e^{-\frac{\beta}{2}t^2} \int_0^t e^{\frac{\beta}{2}s^2}\, s \, dW_s.$
Integrating the  SDE \eqref{eq:SelfAttracting} and applying the stochastic Fubini theorem, we conclude that 
\begin{align*}
    X_t &= x  - \beta \int_0^t Z_u\,du + W_t  
  = x -\beta \int_0^t e^{-\frac{\beta}{2}u^2}
 \left( \int_0^u e^{\frac{\beta}{2}s^2} s\, dW_s \right) du + W_t \\[0.5em] 
&= x + \int_0^t 
\left[ 1- \beta s 
e^{\frac{\beta}{2}s^2} \int_s^t e^{-\frac{\beta}{2}u^2}\,du
\right] dW_s
= x + \int_0^t \kappa(t,s)dW_s,
\end{align*}
with $\kappa$ given in \eqref{eq:Volterra}.
\end{proof}
\bibliographystyle{abbrvnat}
\bibliography{ref.bib}
\end{document}